\documentclass[12pt]{amsart}
\usepackage{amsmath,amsthm,amsfonts,amssymb, mathrsfs,color,bm}

\usepackage{graphicx}
\usepackage{url}
\usepackage{color}
\definecolor{maroon}{rgb}{.69,.188,.376}
\definecolor{darkgreen}{rgb}{0,.5,0}
\definecolor{darkblue}{rgb}{0,0,.5}
\definecolor{magenta}{rgb}{1,0,1}
\definecolor{v1}{RGB}{68,1,84}
\definecolor{v2}{RGB}{57,86,140}
\definecolor{craneorange}{RGB}{31,150,139}
\definecolor{v3}{RGB}{31,150,139}
\definecolor{craneblue}{RGB}{255,255,255}

\usepackage[mathscr]{euscript}		

\usepackage{dsfont}

\usepackage{psfrag}			
\usepackage[colorlinks=true]{hyperref}
\hypersetup{pdftex, colorlinks=true, linkcolor=maroon, citecolor=maroon, filecolor=blue,urlcolor=blue}

\calclayout

\numberwithin{equation}{section}

\usepackage[mathscr]{euscript}		

\newtheorem{thm}{Theorem}[section]
\newtheorem{lem}{Lemma}[section]
\newtheorem{prop}{Proposition}[section]

\newtheorem{defn}{Definition}[section]
\newtheorem{rem}{Remark}[section]
\newtheorem{ass}{Assumption}[section]
\theoremstyle{definition}

\newcommand{\cs}{\mathscr S}

\numberwithin{equation}{section}

\newcommand{\be}{\begin{equation}}
\newcommand{\ee}{\end{equation}}

\newcommand{\bes}{\begin{equation*}}
\newcommand{\ees}{\end{equation*}}

\newcommand{\mP}{\mathbb{P}}

\newcommand{\mt}{\boldsymbol{\eta}}

\newcommand{\X}{\mathbf{X}}
\newcommand{\RR}{\mathbf{R}}
\newcommand{\M}{\mathbf{M}}

\newcommand{\R}{\mathbb{R}}

\newcommand{\Z}{\mathbb{Z}}
\newcommand{\V}{\text{V}}
\newcommand{\HH}{\text{H}}

\newcommand{\N}{\mathbf{N}}

\newcommand{\E}{\mathbb{E}}

\newcommand{\bean}{\begin{eqnarray*}}
\newcommand{\eean}{\end{eqnarray*}}

\newcommand{\e}{\boldsymbol{\eta}}

\newcommand{\car}{\mathcal{R}}

\newcommand{\sfixed}{\mathscr{S}^1(a;t)}
 
\newcommand{\mf}{\mathbf}

\newcommand{\uv}{\mathbf u}
\newcommand{\W}{\mathbf W}

\newcommand{\vv}{\mathbf v}

\begin{document}

\title[]{Sausage Volume of the Random String and Survival in a medium of Poisson Traps}
\author{Siva Athreya \and Mathew Joseph \and Carl Mueller}
\address{Siva Athreya\\ International centre for theoretical Sciences\\Survey No. 151, Shivakote,\\
Hesaraghatta Hobli,\\ Bengaluru - 560 089 \\ and \\ Statmath Unit\\ Indian Statistical Institute\\ 8th Mile Mysore Road\\ Bangalore 560059
} \email{athreya@isibang.ac.in or athreya@icts.res.in}

\address{Mathew Joseph\\ Statmath Unit\\ Indian Statistical Institute\\ 8th Mile Mysore Road\\ Bangalore 560059
} \email{m.joseph@isibang.ac.in}

\address{Carl Mueller \\Department of Mathematics, University of Rochester Rochester, NY  14627
}
\email{carl.e.mueller@rochester.edu}


\begin{abstract}
We provide asymptotic bounds on the survival probability of a
moving polymer in an environment of Poisson traps. Our model for
the polymer is the vector-valued solution of a stochastic heat
equation driven by additive spacetime white noise; solutions
take values in $\R^d, d \geq 1$. We give upper and lower bounds
for the survival probability in the cases of hard and soft
obstacles. Our bounds decay exponentially with rate proportional
to $T^{d/(d+2)}$, the same exponent that occurs in the case of
Brownian motion. The exponents also depend on the length $J$ of
the polymer, but here our upper and lower bounds involve
different powers of $J$.

Secondly, our main theorems imply upper and lower bounds for the
growth of the Wiener sausage around our string. The Wiener
sausage is the union of balls of a given radius centered at
points of our random string, with time less than or equal to a
given value.

\end{abstract}

\maketitle


\section{Introduction}

The model of particles performing random diffusive motion in a region containing randomly located traps is known as the trapping problem (see \cite{hol-weiss} for review). Particle motion is typically Brownian motion in $\R^d$ or a random walk in $\Z^d$.  The traps are placed in a Poissonian manner and the particle gets annihilated on encountering a trap. The main question of interest in such models is the ``Survival Probability'' of the particle. We refer the reader to \cite{szn98} and references there in for a review of the problem of Brownian motion among Poissonian obstacles, to \cite{konig} and references there in for a review of  the problem of a random walk in a random potential and to \cite{ads} for a review of Random walks among mobile and immobile traps.

There is an extensive literature about such trapping problems, see the references in the preceding paragraph. These results often depend on refined estimates for the eigenvalues of the Laplacian or potential theory. On the other hand, the process we consider takes values in function space, and we found it impossible to analyze the situation using existing techniques. Indeed, carrying over finite-dimensional potential theoretic arguments to the infinite dimensional case is often difficult or impossible. We will consider a Gaussian process, but we did not find any Gaussian tools which were relevant to trapping problems.

In this article we will study the annealed survival probability of a random string in a Poissonian trap environment. Let  $(\Omega, {\mathcal F}, {\mathcal F}_t, \mP_0)$ be a filtered  probability space on which $\dot{\mf W}=\dot{\mf W}(t,x)$ is a $d$-dimensional random vector whose components are i.i.d. two-parameter white noises  adapted to ${\mathcal F}_t$. We consider a {\it random string} $\mf u(t,x) \in \mathbb{R}^d$, which is the solution to the following stochastic heat equation (SHE) 
\begin{equation}
\label{eq:she}
\begin{split}
\partial_t {\mf u} (t,x) &=\frac{1}{2}\,\partial_x^2{\mf u}(t,x) + \,\dot {\mf  W}
(t,x)  \\
\mf u(0,x)&=\mf u_0(x)
\end{split}
\end{equation}
on the circle $x\in [0,J]$, having endpoints identified, and $t \in [0,T].$ The initial profile $\mf u_0$ is assumed to be continuous.  Note that we will use boldface letters to denote vector-valued quantities.

We will be interested in the evolution of the random string in a field of obstacles centered at points coming from an independent  Poisson point process. More precisely, let $(\Omega_1,\mathcal{G},\mP_1)$ be a second
probability space on which is defined a Poisson point process {$\mt$} with intensity $\nu$ given by 
 \begin{equation*}
\mt(\omega_1) = \sum_{i\geq 1}\delta_{ \boldsymbol{\xi}_i(\omega_1)},\quad 
       \omega_1\in\Omega_1,
\end{equation*}
with points $\{\boldsymbol{\xi}_i(\omega_1)\}_{i\ge 1}\subset \R^d$. 


The obstacles will be formed via a potential $\text{V} : \R^d \times\Omega_1
\rightarrow [0,\infty]$
$$ \V(\mathbf{z},\e) = \sum_{i \geq 1} \HH(\mathbf{z}-\boldsymbol{\xi_i}),$$
where $\HH: \R^d \rightarrow [0,\infty]$ is a non-negative, measurable function whose support of $\HH$  is contained in the {\it closed} ball $B(\mathbf 0,a)$ of radius $0<a\le 1$ centered at $\mathbf 0$. 

We will  work in the product space $
(\Omega\times\Omega_1,
\mathcal{F}\times\mathcal{G},\mP_0\times\mP_1)$ along with the filtration
$(\mathcal{F}_t\times \mathcal{G})_{t\ge0}$. 
We will write $\E$ for the expectation with respect to
$\mP:=\mP_0\times\mP_1$,
and $\E_i$ for the expectation with respect to $\mP_i$ for $i=0,1$.  Our main quantity of interest is the quenched and the annealed survival probabilities given by
\begin{equation*}
\begin{split}
S_{T, \eta}(\omega_1) &= \E_0\left[\exp\left(-\int_0^T \int_{0}^J\V\Big(\uv
(s,x),\e(\omega_1)\Big)dx
ds\right)\right], \text{ and }\\
S_{T} &=  \E\left[\exp\left(-\int_0^T \int_{0}^J \V\left(\uv(s,x),\e\right)dx ds\right)
\right]
\end{split}
\end{equation*}
respectively. 
Sometimes we will write $S_T^{\HH, J,\nu}$ and $S_{T,\e}^{\HH, J,\nu}$ to emphasize the dependence on $\HH, J,\nu$. 
\subsection{Main Result}
Our first result on {\it hard} obstacles, i.e. the string is killed immediately on contact and the only way it can survive is to avoid them.

\begin{thm}[Hard obstacles]
\label{thm1} 
Consider the solution to \eqref{eq:she} with $d\ge 2$ and $J\ge 1$, and let
$\nu$ and $a$ be as above.  Then the following hold in the case $\HH(\cdot) \equiv \infty \cdot \mathbf{1}_{B(\mathbf{0}, a)}$
\begin{enumerate}
\item (Lower bound)
 There exist positive constants $C_0, C_1, C_2$ independent of $T, J$ such that for  $T\ge C_0J^{2+\frac{d}{2}}$ 
\be \label{lb:h}
S^{\textnormal{H}, J, \nu}_{T} \ge  C_1\exp\left(-C_2\left(\frac{T}{J}\right)^{\frac{d}{d+2}}\right). 
\ee

\item (Upper bound)
There exist positive constants $C_3, C_4$ independent of $T, J$ such that for all $T >0, J \geq 1$
\be
\label{ub:h}
S^{\textnormal{H}, J, \nu}_{T} \le  C_3\exp\left(-\frac{C_4}{1+|\log J|}\left(\frac{T}{J^2}\right)^{\frac{d}{d+2}}\right). 
\ee
\end{enumerate}
\end{thm}

In the case of hard obstacles we immediately see that the survival of the string is only possible if the string avoids the traps. Thus the``sausage of radius $a$ around string up to time $T$'' should be devoid of traps. Indeed it  is easy to  check using standard properties of the Poisson random variable that  
\be \label{eq:pf:hard} S_T^{\HH,J,\nu} = \E\exp\left( -\nu \left|\mathscr{S}^J_{T}(a)\right|\right),\ee
where 
\be \label{eq:u:s} \mathscr{S}^J_{T}(a) =  \mathop{\bigcup}_{\substack{0 \leq s\leq T,\\ 0 \leq y \leq J}} \left\{ \uv(s,y) + B(\mathbf{0},a) \right\},\ee
is the sausage of radius $a$ around $\uv$. Thus Theorem \ref{thm1} also provides bounds on the exponential moments of the volume of the sausage of radius $a$ around the string up to time $T$.

We next turn our attention to the case of {\it soft} obstacles, i.e. $\HH$ does not take the value $\infty$. We make the following specific assumptions on $\HH$. 
\begin{ass} \label{ass} There is a $\mathscr C>0$ such that $\HH(\mathbf x)\ge \mathscr C\cdot \mathbf{1}_{B(\mathbf{0},\frac{a}{2})}(\mathbf{x})$.
\end{ass}
We note that under this assumption there is a positive probability of survival even if the string interacts with the obstacle environment. We are now ready to state our result in this setting.

\begin{thm}[Soft obstacles]
\label{thm2} 
Consider the solution to \eqref{eq:she} with $d\ge 2$ and $J\ge 1$, let
$\nu>0$ be as above and let $\HH$ be a soft obstacle satisfying Assumption \ref{ass}. Then 
\begin{enumerate}
\item (Lower bound)
 There exist positive constants $C_0, C_1, C_2$ independent of $T, J$ such that for  $T\ge C_0J^{2+\frac{d}{2}}$ 
\be \label{lb:s}
S^{\textnormal{H}, J, \nu}_{T} \ge  C_1\exp\left(-C_2\left(\frac{T}{J}\right)^{\frac{d}{d+2}}\right). 
\ee
\item (Upper bound)
Fix $\beta>0$. There exist positive constants $C_3, C_4$ independent of $T, J$ such that for all $T >0, J \geq 1$
\be \label{ub:s}
S^{\textnormal{H}, J, \nu}_{T} \le  C_4\exp\left(-\frac{C_5}{J^{3+\beta}(1+|\log J|)}\left(\frac{T}{J^2}\right)^{\frac{d}{d+2}}\right). 
\ee
\end{enumerate}
\end{thm}

We conclude this sub-section with a few remarks.
\begin{rem} Though the bounds do not match with regard to the exponents of $J$, they do match in that of $T$.
\begin{enumerate}
  \item[(i)] If we set $J=1$ in (\ref{lb:h})and (\ref{ub:h}) or in (\ref{lb:s}) and (\ref{ub:s}) then for large enough $T>0$, in both the hard and soft obstacle cases we have 
$$C_1\exp\left(-C_2 T^{\frac{d}{d+2}}\right) \le S^{\textnormal{H}, 1, \nu}_{T} \le  C_3\exp\left(-C_4 T^{\frac{d}{d+2}}\right),
$$
for some constants $C_1,C_2,C_3,C_4 >0.$
  \item[(ii)] Due to \eqref{eq:pf:hard}, Theorem \ref{thm1} immediately gives us bounds on $\E\exp\left( -\nu \left|\mathscr{S}^J_{T}(a)\right|\right)$. Further, as observe above at $J=1$, from Theorem \ref{thm1} that  we have for sufficiently large $T >0$
$$ C_1\exp\left(-C_2 T^{\frac{d}{d+2}}\right) \leq \E\exp\left( -\nu \left|\mathscr{S}^1_{T}(a)\right|\right)\leq C_3\exp\left(-C_4 T^{\frac{d}{d+2}}\
\right), $$
for some constants $C_1,C_2,C_3,C_4 >0.$ The exponent of $T$ matches that of the asymptotics of the volume of the Brownian Sausage. Since in our case ${\mf u}$ is the solution of a stochastic PDE, this seems to be a new result of independent interest.

\item[(iii)] The constants mentioned in Theorem \ref{thm1} and Theorem \ref{thm2} are all independent of $T,J$ but do depend on $\nu, a,$ and $\mathscr C$. The bound $0\leq a \leq 1$ is used for technical convenience.

\end{enumerate}

\end{rem}

\subsection{Overview of Proof}

We will say that ${\mf u}$ is  a solution  to \eqref{eq:she} if it satisfies,
\begin{equation}
\label{eq:weak-form}
\mf u(t,x)=\int_{0}^{J}G^{(J)}(t,x-y)\mf u_0(y)
dy
  + \int_{[0,t]\times[0,J]}G^{(J)}(t-s,x-y)\mathbf{W}(dsdy),
\end{equation}
where $ [0,J]$ is the circle with endpoints identified and $G^{(J)}:\R_+\times[0,J]\rightarrow\R$ is the fundamental solution of 
the heat equation 
\begin{equation*}
\begin{split}
\partial_t G^{(J)}(t,x) 
&=\frac{1}{2}\,\partial_x^2 G^{(J)}(t,x), \\
G^{(J)}(0,x)&=\delta(x).
\end{split}
\end{equation*}
Furthermore, the final integral in \eqref{eq:weak-form} can be regarded as
either a Wiener integral or a white
noise integral in the sense of Walsh \cite{wals}.

The first reduction in the proof is to reduce to the case $J=1.$ We will do this by deriving a scaling relation for $S_{T}^{\HH, J,\nu}$. Consider 
\[\vv(t,x):= J^{-\frac12} \uv(J^2 t, Jx),\]
defined for $x\in [0,1]$ with endpoints identified, and $t\in [0, TJ^{-2}]$. The initial profile is $\vv(0,x)= \vv_0(x) = J^{-1/2} \uv_0(J x)$. It was proved in Lemma 2.2 of \cite{athr-jose-muel} that $\vv$ satisfies 
\begin{equation*}
\begin{split}
\partial_t \vv &=\frac12 \partial_x^2 \vv +\dot{\widetilde \W}, \quad t \in [0, TJ^{-2}]\\
\vv(0,x) &= \vv_0(x), \quad x \in [0,1]
\end{split}
\end{equation*}
for some other white noise $\dot{\widetilde \W}$. Now it is easily checked
\begin{align*}
S_{T}^{\HH, J,\nu} & = \E\left[\exp\left(-\int_0^T \int_{0}^J \sum_{i\ge 1} \HH\left(\uv(s,x)-\boldsymbol{\xi}_i\right) dx ds\right)
\right] \\
& = \E\left[\exp\left(-\int_0^{\frac{T}{J^2}} \int_{0}^1 \sum_{i\ge 1} J^3\HH\left(J^{\frac12} \left(\vv(\tilde s,\tilde x)-\frac{\boldsymbol{\xi}_i}{J^{\frac{1}{2}}}\right)\right) d\tilde x d\tilde s\right)
\right]
\end{align*}
Define 
\be \label{eq:scalingrelation}
\widetilde \HH(\cdot) := J^3 \HH(J^{\frac12}\cdot), \qquad \tilde{\boldsymbol{\xi}_i} := \frac{\boldsymbol{\xi}_i}{J^{\frac12}}, \qquad \mbox{ and } \qquad \widetilde\nu := \nu J^{\frac{d}{2}}.\ee
It is easily seen that $\widetilde \HH$ is supported in the ball $B(\mathbf{0}, aJ^{-\frac12})$. The points $\tilde{\boldsymbol{\xi}}$  form a Poisson point process of intensity $\widetilde{\nu}$. 
Setting $\widetilde T= TJ^{-2}$ we obtain 
\be \label{eq:scaling}
S_{T}^{\HH, J,\nu} =  S_{\widetilde{T}}^{\widetilde{\HH}, 1,\widetilde\nu}.\ee

\begin{rem}[{\bf Important}] \label{rem:scaling} 
For the rest of article will focus on $J=1$, so \eqref{eq:weak-form} becomes 
\[ \uv(t,x) = \int_0^1 G(t,x-y) \uv_0(y) \, dy +\N(t,x),\]
where
\[ \N(t,x) = \int_{[0,t]\times [0,1]} G(t-s, x-y) \W(ds dy)\]
is the noise term. For simplicity of notation, we have also removed the superscipt in $G^{(1)}$. We will work with $S_T^{\textnormal{H}, 1, \nu}$ and finally use the scaling relation \eqref{eq:scaling} to obtain the bounds for  $S_T^{\textnormal{H}, J, \nu}$
\end{rem}

The strategy for proving the lower bound for survival probability in Theorem \ref{thm1} Theorem \ref{thm2} is the same as that of Random walks or Brownian motions moving in a field of Poisson Traps. This is obtained by identifying an optimal configuration for the traps ${\mathbf \xi}$. The configuration being one which has an area free of traps in a ball of radius $\alpha_T$ around the origin and the string under this potential is forced to stay inside this ball till time $T$. The probability of first event is of order $\exp(-C_1 (\alpha_T+a)^d)$.  To calculate the second probability we decompose the string in to two components namely {\it center of mass} and {\it radius} of $\uv$ respectively. More precisely
\begin{equation}\label{eq:xr}\begin{split}
\mathbf X_t &= \int_0^1\uv(t,x)\, dx, \mbox{ (Center of Mass) }\\
\mathbf R_t &= \sup_{x\in [0,1]} \left|\uv(t,x)-\X_t\right|, \mbox{ (Radius). }
\end{split}
\end{equation}
We show that $\mathbf X_t$ and $\mathbf R_t$ are independent. 
Then separately we compute the probabilities of $|\mathbf X_t|\le\frac{\alpha_T}{2},\; t\le T$ and $\mathbf R_t\le\frac{\alpha_T}{2},\; t\le T$, and show that these are bounded by $\exp(-C_2(\frac{T}{\alpha_T^2})).$ Optimising over choice of $\alpha_T$ and the scaling relations discussed above yield the results. We present the details in Section \ref{sec:lb}.

Unlike the lower bound, the  proof of upper bound differs from the classical setting of random walks or that of Brownian motion. 
Following Remark \ref{rem:scaling} we first obtain an upper bound for $S_T^{\HH, 1, \nu}$. Recall 
\be \label{eq:ub:1} S_T^{\HH,1,\nu} = \E\exp\left( -\nu \left|\mathscr{S}^1_{T}(a)\right|\right),\ee
where $\mathscr{S}^1_{T}(a)$ is the sausage of radius $a$ around $\uv$, that is 
\bes \mathscr{S}^1_{T}(a) =  \mathop{\bigcup}_{\substack{0 \leq s\leq T,\\ 0 \leq y \leq 1}} \left\{ \uv(s,y) + B(\mathbf{0},a) \right\}.\ees
 We will explain the strategy for the proof in the case of hard obstacles, the argument for the soft obstacles not being very different. Due to \eqref{eq:ub:1}, an upper bound on the partition function $S_T^{\HH, 1, \nu}$ essentially boils down to obtaining a lower bound on the volume of the sausage $\mathscr{S}_T^1(a)$ around $\uv$. 
  For this, we consider the sausage around $\uv(t)=\uv(t,\cdot)$, that is 
\be \label{eq:u:sf}
\sfixed =  \mathop{\bigcup}_{ 0 \leq y \leq 1}  \left\{ \uv(t,y) + B(\mathbf{0},a) \right\},
\ee
so that 
\[ \mathscr{S}^1_{T}(a)=\bigcup_{0\le t\le T}\sfixed.  \]
We will identify times at which $\sfixed$ do not intersect, so the sum of the volumes of these fixed time sausages will provide the desired lower bound.

 We will consider a set of stopping times $\tau_i$ (see \eqref{eq:tau}) such that the center of mass at these time points, ${\mathbf X}_{\tau_i}$ are separated by at least $4\Lambda$ from each other (where $\Lambda$ is suitably chosen see Lemma \ref{lem:r:unif}). 
Using known results on the volume of the Wiener sausage, we show that the number of $\tau_i$ before time $T$ should be of order $T^{\frac{d}{d+2}}$ (see Lemma \ref{lem:nt:tail}). Now, let 
\[ \N (s,t; x):= \int_{[s,t]\times[0,1]} G(t-r,x-y) \W (dr dy),\]
which represents the noise term from time $s$ to $t$. Then $\N (s,t)$ will represent the function from $x\in[0,1]$ to $\R^d$. For $s<t$, we use the Markov property for $\uv $ to write
\[ \uv(t) = G_{t-s}*\uv(s) + \N(s,t).\]
If $s\ll t$ then the first term is almost a constant because of the smoothening effect of the Laplacian. The volume of the sausage around $G_{t-s}*\uv(s)$ will then be approximately $a^d$.
We show in Lemma \ref{lem:ti}, using the independence of $\X_t$ and $\RR_t$, that with probability $\ge \frac12$ the range (see \eqref{eq:range} for precise definition) of $\N(s,t)$ is at most $\Lambda$ and the volume of the sausage of radius $a$ around $\N(s,t)$ is of order at least $a^{d-2+\epsilon}$. Consequently there is a subset $\{T_i\}$ of $\{\tau_i\}$ of positive density, such that the sausages around $\uv(T_i)$ are disjoint and have volume $a^{d-2+\epsilon}$ each. This gives a lower bound of $a^{d-2+\epsilon} T^{\frac{d}{d+2}}$ on the volume $\mathscr{S}_T^1(a)$.

As mentioned earlier, our strategy for the lower bound of the annealed survival
probability is the same as for a random walk or Brownian motion surviving in a
field of random traps. The upper bound is different, since the usual strategies
depend on potential theory and eigenvalues of the Laplacian, and both of these
are much harder to study for infinite dimensional processes such as the random
string. We are forced to go back to first principles, which perhaps explains the
fact that our upper and lower bounds do not completely match.  It is also important to note  that while the upper and lower bounds match for the case $J=1$, the scaling relations in \eqref{eq:scalingrelation} imply that they don't carry over to the general case via space-time scaling. 

The quenched survival probability is of keen interest and will be the focus of future work. Here the geometry of the string and its topology will come into play.  We did not explore large deviations for the volume of the sausage as there was no immediate ergodicity to establish a limiting value of a Lyapunov exponent.

\textbf{Convention} We will use $C$ to denote constants whose value might change from line to line. Sometimes we will indicate dependence of constants on parameters by putting the parameters in parentheses, for example $C(d), C(\nu, d)$ etc. The notation $C_1, C_2,\cdots$ will be used to denote constants whose value remain fixed throughout a lemma, proposition, theorem etc. Such constants might be used later in which case it will be clear from the context.

\textbf{Acknowledgement:} S.A. research was partially supported by  the CPDA grant from the Indian Statistical Institute and the Knowledge Exchange grant from the International Centre for Theoretical Sciences, C.M. research was partially supported by Simons Collaboration Grant 513424, and  M.J. research was partially supported by Serb Matrics grant MTR/2020/000453 and a CPDA grant from Indian Statistical Institute.

\section{Proof of the lower bound in Theorems \ref{thm1} and \ref{thm2}} \label{sec:lb}

As indicated above we will use the same strategy for the lower bound for the survival probability for both the {\em hard} and {\em soft} obstacle case.  We need a couple of technical results before we begin the proof.

The following lemma is crucial.

\begin{lem} \label{lem:cmrind} 
With $\X_t$ and $\RR_t$ as in \eqref{eq:xr}, we have
\begin{itemize}
  \item[(a)]$\X_t$ is a standard Brownian motion starting at $\int_0^1\uv_0(x) dx$.  
\item[(b)] $\X_t$ and $\RR_t$ are independent.
\end{itemize}
\end{lem}
\begin{proof}
It is easily checked that 
\[ \X_t = \int_0^1 \uv_0(x)dx +\int_0^t \W(dyds)\]
is a standard Brownian motion, and 
\[ \uv(t,x) -\X_t = \int_0^1 \left[ G(t,x-y)-1\right]\uv_0(y) dy +  \int_{[0,t]\times[0,1]}\left[ G(t-s, x-y)-1\right] \W (ds dy),\]
where $G=G^{(1)}$ is the heat kernel on the unit circle. Both these processes are Gaussian. The components of $\X_t$ and $\uv(t,x)-\X_t$ are uncorrelated since
\[ \int_0^t \int_0^1\left[ G(t-s, x-y)-1\right] \, dy ds =0.\]
The second part of the lemma immediately follows. 
\end{proof}
We will also need 
\begin{prop}\label{prop:u:alpha} Assume $\sup_x |\uv_0(x)|\le \frac{\alpha}{2}$.  Then there are constants $0<C_0<1$ and $K_0>0$ such that for all $\alpha\ge K_0$
\be \label{eq:u:alpha} \mP_0\left(\sup_{\stackrel{s\le \alpha^2}{x\in [0,1]}} \left|\uv(s,x)\right|\le \alpha,\; \sup_{x\in [0,1]}\left|\uv(\alpha,x)\right|\le \frac{\alpha}{2} \right) \ge C_0.\ee
\end{prop} 
\begin{proof}
Let us consider first the case $\uv_0\equiv {\mf 0}$.  From the previous lemma 
\[\mP_0\left(\sup_{\stackrel{s\le \alpha^2}{x\in [0,1]}} \left|\uv(s,x)\right|\le \frac{\alpha}{2}\right)\ge \mP_0\left(\sup_{s\le \alpha^2} |\RR_s|\le \frac{\alpha}{4}\right) \mP_0\left(\sup_{s\le \alpha^2} |\X_s|\le \frac{\alpha}{4}\right). \]
Since the last term is a positive constant independent of $T$, it is enough to show that there is a $K_0>0$ such that 
\[ \sup_{\alpha\ge K_0}\;\mP_0\left(\sup_{s\le \alpha^2} |\RR_s|>\frac{\alpha}{4}\right) <1.  \]
Now 
\begin{align} \label{eq:R:tail}
\mP_0\left(\sup_{s\le \alpha^2} |\RR_s|>\frac{\alpha}{4}\right) =\mP_0\left(\sup_{\stackrel{s\le \alpha^2}{x\in [0,1]}} |\uv(s,x) -\X_s|>\frac{\alpha}{4}\right). 
\end{align}
By splitting the time interval into subintervals of length $1$ we have the bound 
\begin{equation} \label{eq:u-x}
\begin{split}
& \mP_0\left(\sup_{\stackrel{s\le \alpha^2}{x\in [0,1]}} |\uv(s,x) -\X_s|>\frac{\alpha}{4}\right)\\
& \le \sum_{k=0}^{[\alpha^2]+1} \mP_0\left(\sup_{\stackrel{s\in [k,k+1]}{x\in [0,1]}} |\uv(s,x) -\X_s|>\frac{\alpha}{4}\right)\\
&\le \sum_{k=0}^{[\alpha^2]+1} \mP_0\left( |\uv(k,0) -\X_k|>\frac{\alpha}{8}\right) \\
&\qquad \qquad+\sum_{k=0}^{[\alpha^2]+1} \mP_0\left(\sup_{\stackrel{s\in [i,i+1]}{x\in [0,1]}} |\left[\uv(k,0)-\X_k\right] -\left[\uv(s,x)-\X_s\right]|>\frac{\alpha}{8}\right)
\end{split}
\end{equation}
Using the standard Fourier decomposition of $G(t,x)$ (see Section 3 of \cite{athr-jose-muel}) one obtains that each coordinate of $\uv(i,0) -\X_i$ has variance
\begin{align*}
 \int_0^k \int_0^1 \left[G(s,y)-1\right]^2\, dy ds 
&= \int_0^k \int_0^1 G^2(s,y) \, dy ds -k\\
&=\int_0^k \sum_{l\ge 1} \exp\left(-(2\pi l)^2 s \right) \, ds \\
&\le C .
\end{align*}
See Lemma 3.1 of \cite{athr-jose-muel} for details. 
Therefore for large $\alpha$
\[  \sum_{k=0}^{[\alpha^2]+1} \mP_0\left( |\uv(k,0) -\X_k|>\frac{\alpha}{8}\right) \le \exp\left(-C\alpha^2\right).\]

Now we turn to the last term in \eqref{eq:u-x}. Consider the process 
\[ \mathbf{M}(s,x)=\left[\uv(k,0)-\X_k\right] -\left[\uv(s,x)-\X_s\right],\quad s\in [k, k+1],\; x\in [0,1].\]
Note that $\mathbf{M}(k,0)=0$. A quick calculation gives
\begin{align*}
\M(s,x)-\M(s,\tilde x) &= \int_{[0,s]\times[0,1]}\left[G(s-r, \tilde x - y)-G(s-r, x-y)\right] \W(dr dy)
\end{align*}
whose components have variance less than $C|x-y|$ (see Lemma 3.1 of \cite{athr-jose-muel} for details).  Similarly for $k\le s<\tilde s\le k+1$ we obtain
\begin{align*}
\M(s,x)-\M(\tilde s,x)  & = \int_{[0,s]\times[0,1]}\left[G(\tilde s-r, x-y)-G(s-r, x-y)\right] \W (dr dy) \\
&\quad + \int_{[s,\tilde{s}]\times[0,1]} G(\tilde s-r,x-y) \W(dr dy) +\left[\X_s-\X_{\tilde s}\right].
\end{align*}
Following Lemma 3.1 of \cite{athr-jose-muel} we obtain that the components have variance less than $C\sqrt{\tilde s-s}$. Note that $\tilde s\le s+1$ so that the variance of the components of $\X_{\tilde s}-\X_s$ are also bounded by $C\sqrt{\tilde s-s}$. 

Therefore the conditions of Lemma 3.4 of \cite{athr-jose-muel} are satisfied, and we obtain for large $\alpha$
\[ \sum_{k=0}^{[\alpha^2]+1} \mP_0\left(\sup_{\stackrel{s\in [k,k+1]}{x\in [0,1]}}\left |\left[\uv(k,0)-\X_k\right] -\left[\uv(s,x)-\X_s\right]\right|>\frac{\alpha}{8}\right)
 \le \exp\left(-C\alpha^2\right).\]
Returning to \eqref{eq:R:tail} we obtain
\[\mP_0\left(\sup_{s\le \alpha^2} |\RR_s|>\frac{\alpha}{2}\right) \le \exp\left(-C\alpha^2\right) <1, \]
uniformly in $\alpha\ge K_0$ for some $K_0>0$. This completes the proof of the proposition in the case that $\uv_0\equiv {\mf 0}$.

In the general case $\sup_x|\uv_0(x)|\le \frac{\alpha}{2}$, we apply a Girsanov change of measure argument. Consider the measure $\mathbb Q_0$ given by 
\[ \frac{d \mathbb{Q}_0}{d \mathbb{P}_0} =\exp\left(-\int_{[0,\alpha^2]\times[0,1]} \frac{(G_s*\uv_0)(y)\cdot  \W(ds dy)}{\alpha^2} -\frac12\int_0^{\alpha^2}\int_0^1 \frac{\left|G_s*\uv_0)(y)\right|^2}{\alpha^4} \right),\]
where $G_s*\uv_0$ is the convolution of $G(s,\cdot)$ with $\uv_0$. Under the measure $\mathbb Q_0$, 
\[ \widetilde\W (ds dy) = \W(ds dy) +\frac{(G_s*\uv_0)(y)}{\alpha^2} ds dy\]
is a white noise (see \cite{allo}). Moreover 
\be \label{eq:u:g} \uv(t,x) =\left(1-\frac{t}{\alpha^2}\right) (G_t*\uv_0)(x) +\int_{[0,t]\times[0,1]} G(t-s, x-y) \widetilde{\W} (ds dy),\quad 0\le t\le \alpha^2, \ee
and the first term is $0$ at time $t=\alpha^2$. The case of $\uv_0\equiv {\mf 0}$ shows 
\[\mathbb Q_0\left(\sup_{\stackrel{t\le \alpha^2}{x\in [0,1]}} \left|\int_{[0,t]\times[0,1]} G(t-s, x-y) \widetilde{\W} (ds dy)\right|\le \frac{\alpha}{2}\right) \ge \tilde C_0, \]
for some $\tilde C_0>0$. An application of the Cauchy-Schwarz inequality gives 
\begin{align*} &\mathbb P_0\left(\sup_{\stackrel{t\le \alpha^2}{x\in [0,1]}} \left|\int_{[0,t]\times[0,1]} G(t-s, x-y) \widetilde{\W} (ds dy)\right|\le \frac{\alpha}{2}\right) \\
&\ge \mathbb Q_0\left(\sup_{\stackrel{t\le \alpha^2}{x\in [0,1]}} \left|\int_{[0,t]\times[0,1]} G(t-s, x-y) \widetilde{\W} (ds dy)\right|\le \frac{\alpha}{2}\right)^{1/2} \cdot 
\mathbb{E}_{0}\left[\left(\frac{d \mathbb{Q}_0}{d \mathbb{P}_0}\right)^2 \right]^{-1/2} \\
&\ge \tilde C_0^{1/2} \exp\left(-\frac{1}{4}\int_0^{\alpha^2}\int_0^1\frac{\alpha^2}{4\alpha^4} dy ds\right)\\
&= \tilde C_0^{1/2}\exp\left(-\frac{1}{16}\right).
\end{align*}
The first term in \eqref{eq:u:g} is at most $\frac{\alpha}{2}$, and so the above lower bound is also a lower bound for the probability in \eqref{eq:u:alpha}. This completes the proof of the proposition. 
\end{proof}

\begin{proof}[Proof of lower bound in Theorems \ref{thm1} and \ref{thm2}:] Following Remark \ref{rem:scaling} we find the lower bound for $S_T^{\HH, 1,\nu}$. As indicated earlier for the string to survive in a {\em hard} obstacle environment, it must avoid the obstacles. We will use the same strategy of survival for the {\em soft} obstacle as well.
  Let 
\[ \mathcal{O} = \bigcup_{i\ge 1} B(\boldsymbol{\xi}_i,a)\]
be the obstacle set. For $T>0$, let $\alpha\equiv\alpha_T>0$ be a parameter which will be chosen to devise the optimal strategy. Due to the support of $\HH$ being in a ball of radius $a$ one obtains
\begin{equation}\label{eq:st1n} \begin{split} S_T^{\HH, 1,\nu} 
&\ge \mP(\mathcal B_T\cap \mathcal C_T) \\
&= \mP_0(\mathcal B_T) \mP_1(\mathcal C_T),
\end{split}\end{equation}
where 
\begin{align*}
\mathcal B_T &= \left\{\sup_{\stackrel{s\in [0,T]}{x\in [0,1]}}\left|\uv(s,x)\right|\le \alpha\right\},\\
\mathcal C_T &= \Big\{\text{there are no }\boldsymbol{\xi}_i \text{ in the ball } B\left(\mathbf{0},\alpha+a\right)\Big\}.
\end{align*}
It is important to observe here that the above argument does not depend on whether the obstacles are hard or soft. Clearly
\be \label{eq:ct} \mP_1(\mathcal C_T) = \exp\left(-\nu c_d (\alpha+a)^d\right)\ee
for some dimension dependent constant $c_d$. 

We will next estimate $\mP_0(\mathcal B_T)$ by using Proposition \ref{prop:u:alpha}.
Indeed, an application of the Markov property and \eqref{eq:u:alpha}  yields 
\be \label{eq:bt} \mP_0(\mathcal B_T) \ge \mP_0\left(\sup_{\stackrel{s\le \alpha^2}{x\in [0,1]}} \left|\uv(s,x)\right|\le \alpha,\; \sup_{x\in [0,1]}\left|\uv(\alpha,x)\right|\le \frac{\alpha}{2} \right)^{\frac{T}{\alpha^2}} \leq \exp\left(\frac{T}{\alpha^2} \log C_0\right),\ee
for $\alpha^2\ll T.$

Using \eqref{eq:ct} and \eqref{eq:bt} in \eqref{eq:st1n} we have for $\alpha\ge K_0$
\begin{equation*}
\begin{split}
S_T^{\HH, 1,\nu} &\ge  \exp\left(-\nu c_d (\alpha+a)^d\right) \exp\left(\frac{T}{\alpha^2} \log C_0\right) \\
&\ge  \exp\left(-\nu c_d 2^da^d\right) \exp\left(-\nu c_d 2^d\alpha^d+\frac{T}{\alpha^2} \log C_0\right).
\end{split}
\end{equation*}

For general $J\ge 1 $ we use \eqref{eq:scaling}, as well as the fact that $\widetilde a = aJ^{-\frac12},\, \widetilde \nu= \nu J^{\frac{d}{2}}$ to obtain 
\begin{equation*}
\begin{split}
S_T^{\HH, J,\nu} &\ge  \exp\left(-\nu c_d 2^da^d\right) \exp\left(-\nu J^{\frac d2}c_d 2^d\alpha^d+\frac{T}{J^2\alpha^2} \log C_0\right).
\end{split}
\end{equation*}
A simple calculus computation shows that the maximum of the exponent in the second term is attained at $\alpha=C(d,\nu) \left(\frac{T}{J^{2+\frac{d}{2}}}\right)^{\frac{1}{d+2}}$ so that 
\be \label{eq:endlb}S_T^{\HH, J,\nu} \ge  \exp\left(-\nu c_d 2^da^d\right) \exp\left(-C_1(d,\nu)\left(\frac{T}{J}\right)^{\frac{d}{d+2}}\right),\ee
for a constant $C_1(d,\nu)$ independent of $J,\, T$, as long as $\alpha\ge K_0$ or equivalently $T\ge C_2(d,\nu) J^{2+\frac{d}{2}}$. 
\end{proof}

\section{Preliminaries for the upper bounds in Theorems \ref{thm1} and \ref{thm2}}
\label{sec:prelim}
In this section we prove several preliminary lemmas required for the proof. We begin with Section \ref{subsec:ews}, where we define the stopping times $\{\tau_i\}$ precisely and show that there are order of $T^{\frac{d}{d+2}}$ such times in $[0,T]$ with very high probability. In Section  \ref{subsec:sts}, we define the crucial stopping times $\{T_i\}$ at which we will consider the volume of the sausage around $\uv(T_i,\cdot)$. We will choose the $\{T_i\}$ from the $\{\tau_i\}$ so that
\begin{enumerate}
\item $\uv(T_i,\cdot)$ has a larger volume than the sausage of radius $a/2$ around $\N\left(T_{i-1}, T_i\right)$,
\item the volume of the sausage around $\N(T_{i-1}, T_i)$ of radius $\frac{a}{2}$ is at least ${\mathcal{C}}_{\gamma} a^{d-2+\gamma}$, and 
\item the range of $\N(T_{i-1}, T_i)$ has volume less than or equal to $\Lambda$.
\end{enumerate}
Then in Section \ref{subsec:stifp} we show is that there are sufficiently many 
times $\{T_i\}$, and they are far enough apart to ensure that the sausages at these times do not intersect and the gaps between these times have finite mean. Finally we conclude with Section \ref{subsec:esoft} where we prove some estimates needed for the soft obstacle case.

\subsection{Using Estimates of the Wiener Sausage} \label{subsec:ews}
For $ \Lambda> 1$ (to be chosen specifically in Lemma \ref{lem:r:unif}), it is useful to consider the sausage of radius $4 \Lambda$ around the center of mass $\X_T$:
\be \label{eq:x:s} \mathcal{X}_T(4 \Lambda) :=\bigcup_{0\le t\le T}\left\{\X_s + B(\mathbf{0},4 \Lambda)\right\}. \ee
This is the well studied Wiener sausage.
\begin{lem}[\cite{berg-bolt-holl}, \cite{bolt}, \cite{dons-vara}, \cite{szni}] \label{lem:ws:tail}
Let $d\ge 2$ and $\Lambda>1$. There exists $C(d,\Lambda)>0$ such that for $T>0$ we have 
\be \label{eq:ld:vol}
 \mP_0\left(\left|\mathcal X_T(4\Lambda )\right|\le T^{\frac{d}{d+2}}\right) \le  \exp\left(- C(d,\Lambda)T^{\frac{d}{d+2}}\right).
\ee
\end{lem}

Let $\tau_0=0$ and consider consecutive stopping times $\tau_i$ defined as 
\be \label{eq:tau} \tau_{i+1}=\inf\left\{t>\tau_i: \text{dist}\left(\X_t, \bigcup_{k=0}^i \X_{\tau_k}\right)\ge  4 \Lambda\right\}.\ee
Note that $\X_{\tau_{i}}$ is on the boundary of the region $\bigcup_{k=0}^{i-1} B(\X_{\tau_k}, 4 \Lambda)$. Let 
\be \label{eq:nt}
\#(T):= \left|\left\{i\ge 1: \tau_i \le T\right\}\right| 
\ee
be the number of $\tau_i$'s with $i\ge 1$ before time $T$.

\begin{lem} \label{lem:nt:tail}Let $d\ge 2$ and $\Lambda>1$. There exists $C_{d}>0$ such that for $T>0$ we have 
\be \label{eq:nt:tail} \mP_0\left(\#(T)\le \frac{C_dT^{\frac{d}{d+2}}}{\Lambda^d} \right)  \le \exp\left(- C(d,\Lambda)T^{\frac{d}{d+2}}\right),\ee
where $C(d,\Lambda)$ is the same as in Lemma \ref{lem:ws:tail}.
\end{lem}
\begin{proof} We first claim that 
\be \label{eq:b:s}
\bigcup_{i=0}^{\#(T)+1} B\left(\X_{\tau_i}, 8 \Lambda\right) \supset \mathcal X_T(4 \Lambda). 
\ee
Clearly we need to just consider the behavior of the sausage for the time points strictly between $\tau_i$ and $\tau_{i+1}$ for any fixed $i$. By the definition of $\tau_{i+1}$ the path $\X_t,\, \tau_i\le t\le \tau_{i+1}$ is inside $\bigcup_{k=0}^iB(\X_{\tau_k}, 4\Lambda )$. Therefore any such $\X_t$ is within $4\Lambda$ distance of some $\X_{\tau_k}, k\le i$.  It then follows that $B(\X_t, 4 \Lambda) \subset B(\X_{\tau_k}, 8 \Lambda)$. The claim \eqref{eq:b:s} then follows immediately.

Using \eqref{eq:b:s} and the formula for the volume of a $d$ sphere of radius $8\Lambda$ 
\be \label{eq:nt:vol} (\#(T)+2)\frac{ \pi^{\frac d2} (8\Lambda)^{d}}{\Gamma(\frac d2+1)}  \ge \left|\mathcal X_T(4 \Lambda)\right| .  \ee
Therefore 
\begin{align*}
 \mP_0\left(\#(T)\le \frac{C_dT^{\frac{d}{d+2}}}{\Lambda^d}\right)  &\le \mP_0 \left((\#(T)+2) \frac{\pi^{\frac{d}{2}} (8\Lambda)^d}{\Gamma(\frac{d}{2}+1)}\le \frac{C_d16^d \pi^{\frac{d}{2}} T^{\frac{d}{d+2}} }{\Gamma(\frac{d}{2}+1)}\right) \\
 &\le \mP_0\left(\left|\mathcal X_T(4\Lambda)\right| \le \frac{C_d16^d \pi^{\frac{d}{2}} T^{\frac{d}{d+2}}}{\Gamma(\frac{d}{2}+1)}\right) \\
 &\le\exp\left(- C(d,\Lambda) T^{\frac{d}{d+2}}\right),
\end{align*}
if we choose
\[C_d^{-1}=  \frac{16^d\pi^{\frac{d}{2}}}{\Gamma\left(\frac{d}{2}+1\right)}.\]
This completes the proof of the lemma. 
\end{proof}

\subsection{Stopping times for string} \label{subsec:sts}
For a $\R^d$ valued function $\mathbf f$ defined on $[0,1]$, we denote the {\it range} of $\mathbf f$ by 
\be\label{eq:range}
\mathcal{R}(\mathbf f) := \sup_{x, y\in [0,1]} |\mathbf f(x)-\mathbf f(y)|
\ee
We will need the following lemma
\begin{lem}\label{lem:range:conv} Let $\mathbf f:[0,1]\to \R^d$. We have for $t\ge 1$ 
\begin{align*}
\car(G_t*\mathbf f) \le 4d e^{-4\pi^2 t} \left\| \mathbf f - \int_0^1 \mathbf{f}(x) dx \right \|_2\le   4 de^{-4\pi^2 t} \car(\mathbf f)
\end{align*}
\end{lem}
\begin{proof}
We expand each component $\mathbf f $ in the Fourier basis
\[ f_j(x) = \sum_k a^{(j)}_k e^{i2\pi k x} .\]
Then
\begin{align*}
\sup_{x,y} \Big|(G_t*f_j) (x) - (G_t*f_j) (y)\Big| &= \sup_{x,y}\left| \sum_{k\ne 0} e^{-4\pi^2 k^2 t} a_k^{(j)} \left[e^{i2\pi k x} -e^{i2\pi k y}\right]\right|  \\
&\le 4 e^{-4\pi^2 t}\left(\sum_{k \ne 0} \left[a_k^{(j)}\right]^2\right)^{1/2}
\end{align*}
Now observe that $\int_0^1 f_j(x) dx$ is the zeroth Fourier coefficient of $f_j$ so that 
\[\left(\sum_{k \ne 0} \left[a_k^{(j)}\right]^2\right)^{1/2}=  \left\|f_j -\int_0^1 f_j(x) dx\right\|_2 \le \left\| \mathbf f - \int_0^1 \mathbf{f}(x) dx \right \|_2. \]
Clearly $\car(\mathbf f)$ is an upper bound for the right hand side. Finally note that 
\[ \car(G_t*\mathbf f) \le \sum_{i=1}^d \car(G_t* f_j)\]
which gives the factor $d$ in the bound.
\end{proof}
For the rest of this article, we let \\
\fbox{
 \addtolength{\linewidth}{-2\fboxsep}%
 \addtolength{\linewidth}{-2\fboxrule}%
 \begin{minipage}{\linewidth}
  \begin{equation} \label{eq:delta:L}\begin{split}
  \delta &= \frac{a}{100}, \\
 L &= E + 3\vert\log a\vert,
\end{split}
  \end{equation}
where $E$ is a large enough constant.
 \end{minipage}
}

\begin{rem} The constant $E$ chosen above is independent of any of the parameters in the model, e.g. $T, J, a, \nu$. It is chosen so that $L$ satisfies the following: 
\begin{itemize}
\item $4d e^{-4\pi^2 L} \le \delta$ (see Lemma \ref{lem:range:conv}).
\item $L \ge \textnormal{D}+ 2|\log a|$, where $\textnormal{D}$ is the constant appearing in Lemma \ref{lem:vol:n}.
\item $e^{4\pi^2 L}\ge \frac{8C_0 d^{\frac32}}{\delta}$, where $C_0$ is the constant appearing in \eqref{eq:si-ti}.
\item $L$ is large enough so that the last inequality in \eqref{eq:r:ot} holds.
\end{itemize}
\end{rem}

Recall the stopping times $\tau_i$ defined in \eqref{eq:tau}. Let $T_0=0$ and define the sequence of stopping times
\begin{equation}\label{eq:t:s}\begin{split} 
S_{1} &= \left\{ t\ge T_0+L:\; \car\left(G_{t-T_0}*  \uv_0\right) \le \delta\right\}\\
T_1 &=\min\left\{\tau_j: \;\tau_j\ge S_1\right\},\\
 S_2 &= \inf\left\{t\ge T_1+L:\; \car\Big(G_{t-T_1}* \N(T_0, T_1)\Big) \le \delta\right\},  \\
 T_2 &=\min\left\{\tau_j: \;\tau_j\ge S_2\right\}, \\
 S_3 &= \inf\left\{t\ge T_2+L:\; \car\Big(G_{t-T_2}* \N(T_1,T_2)\Big) \le \delta\right\},  \\
 T_3 &=\min\left\{\tau_j:\; \tau_j\ge S_3\right\}, \\
 \vdots &\qquad \vdots \qquad \vdots
\end{split} \end{equation}
The reason for introducing the stopping times $T_i$ will be clear below. Inductively
\begin{align*}
\uv(T_i) & = G_{T_i-T_{i-1}}*\uv\big(T_{i-1}\big) +\N(T_{i-1}, T_i) \\
&= G_{T_i-T_{i-1}}*\left[ G_{T_{i-1} -T_{i-2}}*\uv\big(T_{i-2}\big) +\N(T_{i-2}, T_{i-1})\right] +\N(T_{i-1}, T_i)\\
&= G_{T_i-T_{i-2}} *\uv\big(T_{i-2}\big) +  G_{T_i-T_{i-1}}*\N(T_{i-2}, T_{i-1}) + \N(T_{i-1}, T_i) \\
& \vdots \qquad \vdots \qquad \vdots \\
&= G_{T_i-T_0}*\uv(T_0) +G_{T_i-T_1}*\N(T_0, T_1) +G_{T_i-T_2}*\N(T_1, T_2)+\cdots + \N(T_{i-1}, T_i)
\end{align*}
\begin{defn}For a $\R^d$ valued function $\mathbf f$ on $[0,1]$ we denote
\[ \mathscr S(a; \mathbf{f}) := \bigcup_{0\le y\le 1 } \left\{\mathbf{f}(y) +B(\mathbf{0}, a)\right\}\]
to be the sausage of radius $a$ around $\mathbf f$. 
\end{defn}
The following lemma is crucial for the upper bound. 
\begin{lem}\label{lem:u:disjoint} We have
\be\label{eq:r:u-n}
 \left|\car\big(\uv(T_i)\big) -\car\big(\N(T_{i-1}, T_i)\big)\right|  \le 2\delta.
\ee
and 
\be \label{eq:vol:u-n}
\left|\cs^1\left(a; T_i\right) \right| \ge \left|\cs\Big(a/2;\; \N\left(T_{i-1}, T_i\right)\Big) \right|.
\ee
\end{lem}
\begin{proof}
Denote by 
\[ \boldsymbol{\mathscr G} :=G_{T_i-T_0}*\uv(T_0) +G_{T_i-T_1}*\N(T_0, T_1) +G_{T_i-T_2}*\N(T_1, T_2)+\cdots + G_{T_i-T_{i-1}}*\N(T_{i-2}, T_{i-1}).\]
It is easily checked that when $\mathbf{f}=\mathbf{g}+\mathbf{h}$ then $|\car(\mathbf{f})-\car(\mathbf{g})|\le \car(\mathbf{h})$. Therefore, with our choice of  $\delta$ and $L$, and by using Lemma \ref{lem:range:conv}, we obtain
\bes
 \left|\car\big(\uv(T_i)\big) -\car\big(\N(T_{i-1}, T_i)\big)\right|  \le \car(\boldsymbol{\mathscr G})\le  \sum_{i\ge 1} \delta^i  \le 2\delta.  
 \ees
To prove \eqref{eq:vol:u-n} , note that
\[ \uv(T_i, x) = [\boldsymbol{\mathscr G}(0) + \N\left(T_{i-1}, T_{i}; x\right)]  +[\boldsymbol{\mathscr G}(x) - \boldsymbol{\mathscr G}(0)],\]
and so the ball of radius $a/2$ around $[\boldsymbol{\mathscr G}(0) + \N\left(T_{i-1}, T_{i}; x\right)] $ is contained in the ball of radius $a$ around $\uv(T_i, x)$ (note $\car(\boldsymbol{\mathscr{G}} )\le a/2$ by our choice of $\delta$). Consequently the sausage of radius $a/2$ around $\N\left(T_{i-1}, T_{i}\right)$ has a smaller volume than the sausage of radius $a$ around $\uv(T_i)$. 
\end{proof}
 
 \begin{rem} \label{rem:plan} Equations \eqref{eq:r:u-n} and \eqref{eq:vol:u-n} show that the range of $\uv(T_i)$ is close to that of $\N(T_{i-1}, T_i)$, and a lower bound on the volume of the sausage around $\uv(T_i)$ is given by the volume of a smaller sausage around $\N(T_{i-1}, T_i)$ . As we will see later $\N(T_{i-1}, T_i)$ form a weakly dependent sequence.  In particular, using a weak form of law of large numbers, we will see that there is a subset of $O(T^{\frac{d}{d+2}})$ many $T_i$'s where the volume of the sausage around $\N(T_{i-1}, T_i)$ is large and where the range of $\N(T_{i-1}, T_i)$ is small. Because of \eqref{eq:r:u-n} and \eqref{eq:vol:u-n} the same holds for $\uv(T_i)$ on this subset. The small range guarantees that the sausages at these $T_i$'s are disjoint, and thus a lower bound for $\mathscr S_T^1(a)$ is obtained by adding the volumes of the sausages at the $T_i$'s on this subset.
 \end{rem}

\subsection{Volume of the sausage around $\N(T_{i-1}, T_i)$ and Range of $\N(T_{i-1}, T_i)$}  \label{subsec:vsn}

Our first objective will be to give a lower bound on the probability that the volume of the sausage of radius $a/2$ around $\N(0,t)$ is at least $Ca^{d-2+\epsilon}$ (see Lemma \ref{lem:vol:n}). As indicated in Remark \ref{rem:plan}, we will show later that there are sufficiently many $i$ such that the sausages around $\N(T_{i-1}, T_i)$ have at least this volume. Let 
\be \label{eq:ntxy} \mathbf{N}(t;x,y) = \N(t,x)-\N(t,y),\ee
and write 
\[ \mathbf{N}(t;x,y) = \mathbf{N}^{(1)}(t;x,y) -\mathbf{N}^{(2)}(t;x,y),\]
where 
\begin{align*}
\mathbf{N}^{(1)}(t;x,y)  &= \int_{(-\infty,t]\times[0,1]} \left[G(t-s,x, z) -G(t-s, y,z) \right] \W(ds dz),  \text{ and } \\
\mathbf{N}^{(2)}(t,x,y)  &= \int_{(-\infty,0]\times[0,1]} \left[G(t-s,x, z) -G(t-s, y,z) \right] \W(ds dz)
\end{align*}
It is easy to see that  $\mathbf{N}^{(1)}(t;\cdot, \cdot)$ is a stationary process in $t$. The reason for considering the differences $\N(t,x) -\N(t,y)$ instead of $\N(t,x)$ is that the term 
\[ \int_{(-\infty,t]\times[0,1]} G(t-s, x,z) \W(dz ds) \]
would not be convergent. Note also that the volume of a sausage around $\N(0,t)$ is the same as the volume of the sausage around $\N(t;\cdot, 0)$. The lemma below shows that $\mathbf{N}^{(2)}$ becomes smaller with increasing $t$, and hence the main contribution of $\N(t,x,y)$ comes from $\N^{(1)}(t; x,y)$ for large $t$. The process $\N^{(1)}(t;\cdot, 0)$ behaves locally like Brownian motion (see Lemma \ref{lem:G:space}, and so we can use some techniques (see for example Lemma \ref{lem:mdim}) which work for Brownian motion, to obtain a lower bound on the sausage around $\N^{(1)}(t;\cdot, 0)$.

\begin{lem} \label{lem:u:incr} 
There exists a constant $C_1>0$ such that for any $\lambda>0$
 \be \label{eq:u2:tail} \mP_0\left(\sup_{x,y\in [0,1]} \left|\mathbf{N}^{(2)}(t;x,y)\right|>\lambda\right)  \le 2\exp\left(-\frac{e^t \lambda^2}{C_1}\right). \ee
\end{lem}
\begin{proof}
We first observe 
\begin{align*}
\E_0\left[ \left|\mathbf{N}^{(2)}(t;x,y)\right|^2\right] 
&= \int_{-\infty}^0 \int_0^1 \left[G(t-s,x, z) -G(t-s, y,z) \right]^2dz ds \\ &= \int_{-\infty}^0 ds \sum_{k\ge 1} e^{-k^2(t-s)}  \left|1- e^{i\cdot 2\pi k(x-y)}\right|^2\\
& \le \sum_{k\ge 1} \frac{e^{-k^2 t}}{k^2} \left|1\wedge (k|x-y|)\right|^2 \\
&\le Ce^{-t} |x-y|.
\end{align*}
An application of the Burkholder-Davis-Gundy inequality then gives us 
\[ \E_0\left[ \left|\mathbf{N}^{(2)}(t;x,y)\right|^{2k}\right]  \le C^k (2\sqrt 2)^{2k} k^k e^{-kt}  |x-y|^k\]
for all positive integers $k\ge 1$.
From the argument leading up to inequality (6.9) in \cite{conu-jose-khos} one then obtains
\[ \E_0\left[\sup_{x,y \in [0,1]}\left|\mathbf{N}^{(2)} (t;x,y)\right|^{2k}\right] \le C^ke^{-kt}k^k,\quad k\ge 2. \]
 Therefore there exists a constant $C_1>0$ such that
 \[ \E_0\left[\sup_{x,y \in [0,1]}\exp \left(\frac{\exp(t) \cdot |\mathbf{N}^{(2)}(t;x,y)|^2}{C_1}\right)\right]  \le 2.\] 
 The inequality \eqref{eq:u2:tail} follows immediately from this. 
\end{proof}
The following lemma indicates that the process $\N^{(1)}(t;\cdot, 0)$ behaves locally like Brownian motion. It will be used in Lemma \ref{lem:mdim} below.
\begin{lem} \label{lem:G:space} There are constants $C_1, C_2>0$ such that for all $x, y \in [0,1]$ and all $t\ge 1$ one has
\[ C_1 d(x,y) \le \int_0^t \int_0^1 \left[G(s, x,z) -G(s,y,z)\right]^2 dz ds \le C_2d(x,y),\]
where $d(x,y)$ is the distance between $x$ and $y$ on the torus $\mathbb T=[0,1)$.
\end{lem} 
\begin{proof} Using the Fourier decomposition of $G(s,x)$ we obtain 
\begin{equation}\begin{split}\label{eq:parseval}
 &\int_0^t ds  \int_0^1 dz \left[G(s, x,z) -G(s,y,z)\right]^2  \\
&= C\int_0^t ds \sum_{k\ge 1} \exp\left(-(2\pi k)^2 s\right)\big|1-\exp\left[i(2\pi k) d(x,y)\right]\big|^2  \\
&= C\sum_{k \ge 1} \frac{\left[1- \exp(-(2\pi k)^2 t)\right]}{k^2} \big|1-\exp\left[i(2\pi k) d(x,y)\right] \big|^2.
\end{split}\end{equation}
Now use $|1-e^{iz}|\le 1\wedge |z|$ and $|1-e^{-4\pi^2 k^2 t} |\le 1$ to obtain that the above is less than
\[
 C \sum_{k\ge 1} \frac{1}{ k^2} \big|1\wedge [kd(x,y)]\big|   \le Cd(x,y).\]
The final inequality is obtained by splitting the sum according to whether $k\le d(x,y)^{-1}$ or $k>d(x,y)^{-1}$.

Next we turn to the lower bound. For this  we only consider the sum in \eqref{eq:parseval} for $1\le k \le \frac{1}{2d(x,y)}$. Now $|1-e^{iz}|\ge  C|z|$ for all $z\in [0,\pi]$ and some $C>0$. Therefore since $t\ge 1$
\[ \int_0^t ds  \int_0^1 dz \left[G(s, x,z) -G(s,y,z)\right]^2 \ge C\sum_{k=1}^{[2d(x,y)]^{-1}} \frac{k^2d(x,y)^2}{k^2} \ge Cd(x,y).\]
This completes the proof of the lemma. 
\end{proof}

Before proceeding we recall
\begin{defn} The lower Minkowski dimension of a set $A$ is 
\[\underline{\textnormal{dim}}_M(A)  :=\liminf_{\epsilon\to 0} \frac{\log N_{\epsilon}(A)}{\log(\epsilon^{-1})},\]
where $N_{\epsilon}(A)$ is the minimum number of balls of radius $\epsilon$ needed to cover $A$. 
\end{defn}
We will need
\begin{lem} \label{lem:mdim} For $t\ge 1$
\[ \underline{\textnormal{dim}}_M \left[\textnormal{Range}\big(\mathbf{N}^{(1)}(t; \cdot,0)\big)\right] \ge 2 \;\; \text{a.s.}\]
\end{lem}
\begin{proof} We first recall that for any set $A$ we have 
\[  \underline{\textnormal{dim}}_M(A)\ge {\textnormal{dim}}_H(A), \]
where ${\textnormal{dim}}_H$ is the Hausdorff dimension (see page 115 of \cite{mort-pere}). We use the {\it energy method} (Theorem 4.27 in \cite{mort-pere}) to get a lower bound on the Hausdorff dimension of the range. Let $\mu_t$ be the occupation measure of $\mathbf{N}^{(1)}(t; \cdot,0)$:
\[ \int_{\R^d} f(\mathbf{x}) d\mu_t(\mathbf{x}) = \int_0^1 f\left(\mathbf{N}^{(1)}(t; x,0)\right)dx.\]
We just need to show that for any $0<\alpha<2$
\[ \E_0 \int_{\R^d}\int_{\R^d}\frac{d\mu_t(\mathbf{x}) d\mu_t(\mathbf{y})}{|\mathbf{x}-\mathbf{y}|^\alpha} =\E_0\int_0^1\int_0^1 \frac{dxdy}{|\mathbf{N}^{(1)}(t; x, 0)-\mathbf{N}^{(1)}(t; y, 0)|^\alpha}<\infty.\]
 Now $\mathbf{N}^{(1)}(t, x, 0)-\mathbf{N}^{(1)}(t, y, 0)$ is a Gaussian random variable with mean $0$ and variance
 \[ \int_0^{\infty} \int_0^1 \left[G(s, x,z) -G(s, y,z)\right]^2 dz ds .\]
 From Lemma \ref{lem:G:space} this expression is bounded above and below by a constant multiple of $d(x,y)$. Therefore it follows that 
 \[ \E_0\int_0^1\int_0^1 \frac{dxdy}{|\mathbf{N}^{(1)}(t; x, 0)-\mathbf{N}^{(1)}(t; x, 0)|^\alpha} \le C\int_0^1\int_0^1 \frac{dxdy}{d(x,y)^{\alpha/2}} <\infty,\] 
 as required. Consequently
 \[\int_{\R^d}\int_{\R^d}\frac{d\mu_t(\mathbf{x}) d\mu_t(\mathbf{y})}{|\mathbf{x}-\mathbf{y}|^\alpha} <\infty \quad \text{a.s.} \] 
 It thus follows from Theorem 4.27 in \cite{mort-pere} that $\underline{\textnormal{dim}}_M\left[ \textnormal{Range}\big(\mathbf{N}^{(1)}(t; \cdot,0)\big)\right] \ge 2  \; \text{a.s.}$ as required.
\end{proof}
The lower bound in the lower Minkowski dimension gives a lower bound on the number of balls of radius $a$ which intersect the range of $\N^{(1)}(t;\cdot,0)$, and consequently a lower bound on the volume of the sausage of radius $a$ around $\N^{(1)}(t;\cdot, 0)$. 
\begin{lem} \label{lem:vol:u1} Fix an arbitrary $0<\gamma<1$. There exists a $\tilde{\mathcal C}_{\gamma}>0$ such that for any $t\in \R$ and any $0<a \le 1$
\[ \mP_0\left(\vert \mathscr S\left(a;\;\mathbf{N}^{(1)} (t; \cdot, 0)\right)\vert \ge \tilde{\mathcal C}_{\gamma} a^{d-2+\gamma}\right) \ge \frac{4}{5}.\] 
\end{lem}
\begin{proof} First note that the process $\mathbf{N}^{(1)} (t; \cdot, 0)$ is stationary in $t$. Partition $\R^d$ into cubes of side length $a$. Let $\tilde N_a(\mathbf{N}^{(1)}, t)$ be the number of cubes through which  $\mathbf{N}^{(1)}(t;\cdot,0)$ passes. By Lemma \ref{lem:mdim} we obtain almost surely
\[ \tilde N_a(\mathbf{N}^{(1)}, t) \ge \left(\frac 1a\right)^{2-\gamma} \text{ for all $a$ small enough}. \]
Therefore there exists a positive random variable $A(\omega)$ which is finite almost surely such that 
\[ \tilde N_a(\mathbf{N}^{(1)}, t) \ge A(\omega)\left(\frac 1a\right)^{2-\gamma} \text{ for all }0<a\le 1 . \]
We now choose the largest possible subcollection of these $\tilde N_a(\mathbf{N}^{(1)}, t)$ cubes such that no two cubes are adjacent (that is, share a common edge). Let $ N_a^{*}(\mathbf{N}^{(1)}, t)$ be the number of cubes in this subcollection. Clearly there exists a constant $C_d$ such that almost surely 
  \[  N_a^{*}(\mathbf{N}^{(1)}, t)\ge A(\omega)C_d\left(\frac 1a\right)^{2-\gamma} \text{ for all } 0<a\le 1. \]
 From each of these $N_a^{*}(\mathbf{N}^{(1)}, t)$ cubes choose any point in the range of $\mathbf{N}^{(1)}(t,\cdot,0)$. The union of the balls of radius $a$ around these  points is contained in the sausage of radius $a$ around $\mathbf{N}^{(1)} (t, \cdot, 0)$, so that
 \[ \vert \mathscr S\left(a;\;\mathbf{N}^{(1)} (t, \cdot, 0)\right)\vert \ge A(\omega)\tilde C_d a^{d-2+\gamma} \quad \text{ for all } 0<a\le 1,\]
 for some other constant $\tilde C_d>0$.
 The constant $\mathcal C_{\gamma}>0$ is chosen so that
 \[\mP_0\left(A(\omega) \tilde C_d \ge \tilde{\mathcal C}_{\gamma} \right) \ge \frac 45.\]
 This completes the proof of the lemma. 
 \end{proof}
%
%
%
We now use Lemma \ref{lem:u:incr} on the smallness of $\N^{(2)}(t;\cdot, 0)$ to control the volume of the sausage around $\N(0,t)$.
\begin{lem}\label{lem:vol:n} Fix an arbitrary $0<\gamma<1$.  There are constants ${\mathcal C}_{\gamma} >0$ and $\textnormal{D}>0$ such that for any $0<a\le 1$ and $t\ge \textnormal{D} +2|\log a|$
\[  \mP_0\left(\vert \mathscr S\left(\frac{a}{2};\;\N (0,t)\right)\vert \ge {\mathcal C}_{\gamma} a^{d-2+\gamma}\right) \ge \frac{3}{4}.\]
\end{lem}
\begin{proof} We shall use \eqref{eq:u2:tail}.
\[ \mP_0\left(\sup_{x,y\in [0,1]} \left|\mathbf{N}^{(2)}(t;x,y)\right|>\frac{a}{4}\right)  \le 2\exp\left(-\frac{e^t a^2}{16 C_1}\right). \]
If we choose $\textnormal{D}$ large enough the right hand side above is at most $\frac{1}{20}$. Therefore, 
\[\mP_0\left(\car\left(\mathbf{N}^{(2)}(t;\cdot,0)\right)\le \frac{a}{4} \right) \ge \frac{19}{20}. \]
As a consequence of this we obtain by a similar argument as in Lemma \ref{lem:u:disjoint} that 
\[\big\vert \mathscr{S}\left(\frac{a}{2};\; \N(0,t)\right)\big\vert \ge  \big\vert \mathscr{S}\left(\frac{a}{4};\; \N^{(1)}(t; \cdot, 0)\right)\big\vert. \]
We now use Lemma \ref{lem:vol:u1} to complete the proof. 
\end{proof}

Our second objective in this subsection is (following Remark \ref{rem:plan}) analyze the probability that the range of $\N(0,t)$ is small (see Lemma \ref{lem:r:unif}).

\begin{lem} \label{lem:r:unif} There exists $\Lambda>1$ such that for all $t\ge L$
\be \label{eq:r:n} \mP_0\left(\car(\N(0,t)) \le \Lambda\right) \ge \frac{3}{4} .\ee
\end{lem}
\begin{proof}
Since $\mathscr \N^{(1)}(t; \cdot, 0)$ is stationary in $t$, we can choose $\Lambda>1$ such that 
\[ \mP_0\left(\car \left(\N^{(1)}(t; \cdot, 0)\right) \le \frac{\Lambda}{2}\right) \ge \frac{4}{5}. \]
In the proof of Lemma \ref{lem:vol:n} we have seen that with probability at least $\frac{19}{20}$ one has
\[\car\left(\N^{(2)}(t;\cdot,0)\right)\le \frac{a}{4}.\]
Therefore with probability at least $\frac34$ we have 
\[\car\Big(\N(0,t)\Big) \le \car\Big(\N^{(1)}(t; \cdot, 0) \Big) + \car\Big(\N^{(2)}(t; \cdot, 0) \Big) \le \Lambda. \] 
This completes the proof. 
\end{proof}

\begin{rem} We fix and use a $\Lambda$ as in Lemma \ref{lem:r:unif} for the rest of the article. In particular the $\tau_i$'s defined in \eqref{eq:tau} are defined in terms of this particular choice of $\Lambda$.
\end{rem}

Our final objective is to show that there are sufficiently many $T_i$'s such that $\car\left(\N(T_{i-1}, T_i)\right)\le \Lambda$ and $\cs\Big(\frac{a}{2}; \N\left(T_{i-1}, T_i\right)\Big) \ge {\mathcal{C}}_{\gamma} a^{d-2+\gamma}$ (See Lemma \ref{lem:ti}). Before we proceed, we will need 
\begin{defn} Let $\tilde{\mathcal F}_t $ be the filtration generated by white noise 
\[ \tilde{\mathcal F}_t = \sigma\left\{\W(A\times [r,s]);\; A\subset [0,1],\, 0\le r,s\le t\right\}.\] 
Let $\mathcal G_i$ denote the $\sigma$-algebra generated by the white noise up to time $T_i$:
\[ \mathcal G_i =\left\{ \textnormal{A}\in \mathcal F: \textnormal{A} \cap \left\{T_i \le t\right\}\in \tilde{\mathcal F}_t\right\}\]
Let 
\[\mathcal H_i = \mathcal G_i \vee \sigma\left\{\X_t;\, t\ge 0 \right\}\]
the $\sigma$-algebra generated by the white noise up to time $T_i$ and the center of mass process. 
\end{defn}

\begin{lem}\label{lem:ti} Let $\Lambda$ be as in Lemma \ref{lem:r:unif}. 
We have
\be \label{eq:ti} \mP_0\left(\car\left(\N(T_{i-1}, T_i)\right) \le \Lambda,\; \left |\cs\Big(\frac{a}{2};\,\N(T_{i-1}, T_i))\Big) \right| \ge {\mathcal C}_\gamma a^{d-2+\gamma}\, \bigg \vert\, \mathcal H_{i-1}\right) \ge \frac{1}{2}.
\ee
\end{lem}
\begin{proof} It is enough to show 
\begin{align}
\label{eq:r} \mP_0\left(\car\left(\N\left(T_{i-1}, T_i\right)\right) \le \Lambda \, \bigg \vert\, \mathcal H_{i-1}\right) &\ge \frac34, \\
\label{eq:vol} \mP_0\left( \left |\cs\Big(\frac{a}{2};\,\N\left(T_{i-1}, T_i\right)\Big) \right| \ge {\mathcal C}_\gamma a^{d-2+\gamma}\, \bigg \vert\, \mathcal H_{i-1}\right) &\ge \frac34.
\end{align}
We have from Lemma \ref{lem:r:unif}
\bes
\mP_0\Big(\car\left(\N\left(0,t \right)\right) \le \Lambda\Big) \ge \frac34,
\ees
uniformly in $t\ge L$. Then observe 
\bes \begin{split}
&\mP_0\Big(\car\left(\N\left(T_{i-1}, T_i\right)\right) \le \Lambda \, \bigg \vert\, \mathcal H_{i-1}\Big)\\ & = \int_0^{\infty} \int_{s+L}^\infty \mP_0\left(T_{i-1} \in ds,\; T_i\in dt,\; \car\Big(\N(s,t)\Big)\le \Lambda\, \bigg \vert\, \mathcal H_{i-1}\right) \\
&= \int_0^{\infty} \int_{s+L}^\infty \mathbf{1}\left\{T_{i-1} \in ds,\; T_i\in dt \right\}\cdot \mP_0\left( \car\Big(\N(s,t)\Big)\le \Lambda\right).
\end{split}
\ees
The second equality follows from an argument similar to Lemma \ref{lem:cmrind}, 
In fact the event $\left\{T_{i-1} \in ds,\; T_i\in dt\right\}$ is measurable with respect to the sigma field $\mathcal H_{i-1}$,
 while $\car(\N(s,t))$ depends on
\[\sigma\left(\int_s^{\tilde t} \int_0^1 \left[G_{t-r} (x,z) - G_{t-r}(0,z)\right] \W(dz dr),\;x \in [0,1],\; \tilde t \ge s\right),  \]
which is independent of $\mathcal H_{i-1}$.
From this we obtain \eqref{eq:r}. Similarly, to show \eqref{eq:vol} 
we use Lemma \ref{lem:vol:n},  and integrate over the realizations of $T_{i-1}$ and $T_i$. This completes the proof of the lemma.
\end{proof}


Consequently, due to \eqref{eq:r:u-n} and \eqref{eq:vol:u-n} and using Lemma \ref{lem:ti}, there are sufficiently many $T_i$ such that $\car\left(\uv(T_i)\right) \le \Lambda +2\delta$ and $\left|\cs^1\left(a; T_i\right) \right| \ge {\mathcal{C}}_{\gamma} a^{d-2+\gamma}.$

\subsection{Sufficiently many $T_i$ far apart} \label{subsec:stifp}

\begin{lem} There exists $C_2>0$ such that for all $t\ge L$
\[ \mP_0\left[S_{i+1}- T_{i} >t \; \Big\vert \mathcal{H}_{i-1}\right]  \le \frac{C_2e^{-8\pi^2 t}}{\delta^2}. \]
In particular for any $\eta< 8\pi^2$ there exists $C_3(\eta)>0$ such that
\be \label{eq:siti:exp} \E_0\left[\exp\left(\eta (S_{i+1}- T_{i} )\right)\; \bigg\vert \mathcal{H}_{i-1}\right] \le e^{\eta L}+ \frac{C_3(\eta)}{\delta^2}\ee
\end{lem}
\begin{proof}
Recall from \eqref{eq:t:s}
\[ S_{i+1} =  \inf\left\{t\ge T_{i}+L: \car\Big(G_{t-T_{i}}* \N(T_{i-1}, T_{i})\Big) \le \delta\right\}. \]
The event $\{S_{i+1}- T_{i} >t\}$ implies that $\car\Big(G_{t}* \N(T_{i-1}, T_{i})\Big) > \delta$, and in light of Lemma \ref{lem:range:conv}, it further implies 
\[ \left\|\N(T_{i-1}, T_{i}) - \int_0^1\N(T_{i-1}, T_{i};\, x) dx \right\|_2 >\frac{\delta}{4d} e^{4\pi^2 t}. \]
Using the subscript $j$ to denote the components of $\N(T_{i-1}, T_i)$ it follows that there exists a $1\le j\le d$ such that 
\[\left\|\N_j(T_{i-1}, T_{i}) - \int_0^1\N_j(T_{i-1}, T_{i};\, x) dx \right\|_2 >\frac{\delta}{4d^{\frac{3}{2}}} e^{4\pi^2 t}. \]
We can formally write each component $\W_j$ of the white noise as $\W_j(dydr) = \sum_{k \in \Z} e^{\imath(2\pi k y) } dB_k(r) dy$, where the $B_k$'s are independent standard complex Brownian motions (that is $B_k = \frac{R_k}{\sqrt 2} +\imath \frac{C_k}{\sqrt 2}$ where $R_k, C_k$ are standard real Brownian motions) with $\bar{B_k}= B_{-k}$. This can be seen by integrating both sides with test functions and computing the second moments. Since 
\[ G_t(x,y) =\sum_{l \in \Z} e^{-2\pi^2l^2 t} e^{\imath 2\pi l(x-y)},\]
the $k$th Fourier coefficient of $\N_j(s, \tilde s),\; k\ne 0$ is 
\[ a_k=\int_s^{\tilde s} e^{-2\pi^2 k^2 (\tilde s-r)}  dB_k(r) 
\]
Furthermore we have $\bar{a_k}=a_{-k}$  and $a_k$ is independent of $a_{\tilde k}$ if $\tilde k\ne k, -k$. Now
\[   \text{E}\left[\sum_{k\ne 0} |a_k|^2\right] = \sum_{k\ne 0} \frac{1}{2\pi^2k^2}\left(1- e^{-2 \pi^2 k^2 (\tilde s-s)}\right),\] and so there exist positive constants $C_0,\, C_1$ such that 
\bes
\text{E}\left[\left(\sum_{k\ne 0} |a_k|^2\right)^{\frac12}\right] \le C_0 \quad \text{and} \quad  \text{Var} \left[\left(\sum_{k\ne 0} |a_k|^2\right)^{\frac12}\right] \le C_1
\ees
uniformly in $s$ and $\tilde s$. Therefore 
\be \label{eq:si-ti}\begin{split}
& \mP_0\left[S_{i+1}- T_{i} >t \; \Big\vert \; \mathcal{H}_{i-1}\right] \\
&\le d \cdot\mP_0\left[\left\|\N_j(T_{i-1}, T_{i}) - \int_0^1\N_j(T_{i-1}, T_{i}; x) dx \right\|_2 >\frac{\delta}{4d^{\frac32}} e^{4\pi^2 t} \bigg\vert \; \mathcal{H}_{i-1}\right] \\
&\le \frac{d \cdot C_1}{\left(\frac{\delta}{4d^{\frac32}}e^{4\pi^2 t} -C_0\right)^2}.\end{split}\ee
The second part of the lemma follows from 
\begin{align*}
 \E_0\left[\exp\left(\eta (S_{i+1}- T_{i} )\right)\; \bigg\vert \; \mathcal H_{i-1}\right] 
& \le e^{\eta L} - \int_L^\infty e^{\eta t} \mP_0\left[S_{i+1}- T_{i} >t \; \Big\vert \; \mathcal H_{i-1}\right] \, dt,
\end{align*}
and the above tail bound.
\end{proof}

We conclude this section regarding the spacings of $S_i$ and $T_i$. This will be crucially used on a specific subset of $T_i$'s to show the upper bound.
\begin{lem} \label{lem:eAla}
There is a constant $\widetilde{C}>0$ such that for any $A_4>0$ and $L$ as in \eqref{eq:delta:L} we have  
\[ \mP_0\left(\sum_{i=1}^{A_4 \frac{T^{\frac{d}{d+2}}}{L}} (S_i-T_{i-1}) >  \widetilde{C} A_4T^{\frac{d}{d+2}}\right) \le \exp\left(-\frac{\widetilde{C}A_4}{2} T^{\frac{d}{d+2}}\right). \]
\end{lem}
\begin{proof} With the choice of $\eta=1$ in \eqref{eq:siti:exp} we obtain 
\begin{align*}
\mP_0\left(\sum_{i=1}^{A_4 \frac{T^{\frac{d}{d+2}}}{L}} (S_i-T_{i-1}) >  \widetilde{C}A_4T^{\frac{d}{d+2}}\right) &\le \E_0 \exp\left(\sum_{i=1}^{A_4 \frac{T^{\frac{d}{d+2}}}{L}} (S_i-T_{i-1}) -\widetilde{C}A_4T^{\frac{d}{d+2}}\right)\\
&\le \left(e^L + \frac{C_3}{\delta^2}\right)^{A_4 \frac{T^{\frac{d}{d+2}}}{L}} \exp\left(-\widetilde{C}A_4T^{\frac{d}{d+2}}\right).
\end{align*}
The lemma follows by a large choice of the constant $\widetilde{C}$ above.
\end{proof}

\subsection{Estimates for Soft obstacles} \label{subsec:esoft}
We will need a few lemmas which lead up Proposition \ref{prop:soft}. This is a key proposition that will be used in the proof of the upper bound in Theorem \ref{thm2}.
\begin{prop}\label{prop:n:chain} There is a $C_3>0$ such that for all $s_0\le 1$
\[\mP_0\left(\sup_{s\le s_0}\sup_{x\in [0,1]} \left|\N(0,t+s; x) -\N(0,t; x)\right|>\lambda\right) \le \exp\left(- \frac{C_3^2\lambda^2}{\sqrt{s_0}}\right) \]
uniformly in $t$.
\end{prop}
The proof of the above proposition follows from a sequence of lemmas.  Define for $s,\, t\ge 0$ and $x, y\in [0,1]$
\[ \mathbf Z(t,s; x,y) = \Big\lbrace \N(0, t+s; x) -\N(0,t; x) \Big\rbrace -\Big\lbrace \N(0, t+s; y) -\N(0,t;y)\Big\rbrace \]

\begin{lem} There is a constant $C_1>0$ such that 
\[\mP_0\left(\left|\mathbf Z(t,s; x,y)\right|>\lambda\right) \le \exp\left(- \frac{C_1^2\lambda^2}{\sqrt s \wedge |x-y|}\right) \]
uniformly in $t$. 
\end{lem}
\begin{proof}
We first give an upper bound on $\E_0 \left[\mathbf Z_i^2(t,s; x,y) \right]$, for any fixed coordinate $\mathbf Z_i$ of $\mathbf Z$. This is easily seen to be equal to 
\begin{align} & \label{eq:z:1}\int_0^t \int_0^1 \left[G(t+s-r, x,z) -G(t+s-r,y,z) -G(t-r,x,z) +G(t-r, y,z)\right]^2dz dr  \\
&\label{eq:z:2}\qquad + \int_t^{t+s} \int_0^1 \left[G(t+s-r, x,z) -G(t+s-r, y,z)\right]^2 dz dr
\end{align}
Let us first look at \eqref{eq:z:2}. This is bounded by 
\begin{align*}
\int_0^s dr \sum_{k\ge 1} e^{-(2\pi k)^2r }\Big|1- \exp\left(i(2\pi k)(x-y)\right)\Big|^2 
&\le C\sum_{k\ge 1} \frac{1-e^{-(2\pi k)^2s}}{k^2} \Big[1\wedge |2\pi k(x-y)|\Big]^2 \\
&\le C \sum_{k\ge 1} \frac{1\wedge (2\pi k)^2 s}{k^2} \Big[1\wedge |2\pi k(x-y)|\Big]^2
\end{align*}
In the case that $\sqrt s \le |x-y|$, the above is bounded by 
\[ C\sum_{k=1}^{\frac{1}{2\pi |x-y|}} \frac{k^2s}{k^2}k^2|x-y|^2 +C\sum_{k=\frac{1}{2\pi |x-y|}+1}^{\frac{1}{2\pi \sqrt{s}}} s + C\sum_{k=\frac{1}{2\pi \sqrt{s}}+1}^{\infty} \frac{1}{k^2} \le C\sqrt{s}.\]

In the case that $|x-y|\le \sqrt s$ we obtain a bound 
\[ C\sum_{k=1}^{\frac{1}{2\pi \sqrt s}} \frac{k^2s}{k^2}k^2|x-y|^2 +C\sum_{k=\frac{1}{2\pi \sqrt s} +1}^{\frac{1}{2\pi |x-y|}} \frac{k^2|x-y|^2}{k^2} + C\sum_{k=\frac{1}{2\pi |x-y|}+1}^{\infty} \frac{1}{k^2} \le C|x-y|. \]
Let us next consider the term \eqref{eq:z:1}. This is bounded by
\begin{align*}
& \int_0^t dr \sum_{k \ge 1} \left[\exp\left(- \frac{(2\pi k)^2(t+s-r)}{2}\right)-\exp\left(- \frac{(2\pi k)^2(t-r)}{2}\right)\right]^2 \left[1\wedge \left|2\pi k(x-y)\right|\right]^2 
\\
&\le \int_0^t dr \sum_{k\ge 1} \exp\left(-\pi^2k^2 r\right)\left|1- \exp\left(-2\pi^2k^2 s\right)\right|^2\left[1\wedge \left|2\pi k(x-y)\right|\right]^2 \\
&\le C \int_0^t dr \sum_{k\ge 1} \frac{\left[1\wedge 2\pi^2k^2 s\right]^2}{ k^2}\left[1\wedge \left|2\pi k(x-y)\right|\right]^2 
\end{align*}
Therefore a similar bound as that for \eqref{eq:z:2} holds for \eqref{eq:z:1}. The conclusion of our arguments is that 
\[ \E_0 \left[\mathbf Z_i^2(t,s; x,y) \right]  \le C \left[\sqrt s \wedge |x-y|\right] \]
Since $\mathbf Z(t,s; x,y)$ is Gaussian we obtain the lemma by standard arguments. 
\end{proof}

By similar arguments (see Lemma 3.3 in \cite{athr-jose-muel})  one has 
\begin{lem} There is a constant $\tilde C_1>0$ such that for all $s\le 1$
\[ \mP_0\left(\left|\N(0, t+s;0) -\N(0,t;0)\right|>\lambda\right) \le \exp\left(- \frac{\tilde C_1^2\lambda^2}{\sqrt s}\right)\]
uniformly in $t$. 
\end{lem}

Let $\mathbb D_n$ denote the collection of dyadic points of the form $\frac{k}{2^n}$ in $[0,1]$. For any dyadic point $x\in \mathbb D_n$, we can find a sequence $0=p_0, p_1,\cdots, p_m=x$ of points such that $p_i, p_{i+1}$ are nearest neighbors in some $\mathbb D_k,\, k \le n$, and there are at most $2$ points in any $\mathbb D_k$. Now
\begin{align*}
\N(0, t+s;x) -\N(0, t; x)
 & = \Big[\N(0, t+s;0) -\N(0,t; 0)\Big] +\sum_{i=1}^m \mathbf Z(t,s; p_{i+1}, p_i)
\end{align*}
From this and a chaining argument, similar to that of Lemma 3.4 in \cite{athr-jose-muel} we obtain
\begin{lem} There is a $C_2>0$ such that for all $s\le 1$ 
\[ \mP_0\left(\sup_{x\in [0,1]}\left|\N(0,t+s;x) -\N(0,t;x)\right| >\lambda\right) \le \exp\left(-\frac{C_2^2\lambda^2}{\sqrt s}\right).
\]
uniformly in $t$.
\end{lem}
We clearly have for $s,\, \tilde s \le  1$ 
\begin{align*}
& \Big|\sup_{x\in [0,1]} \left|\N(0,t+s;x) -\N(0,t;x)\right| -\sup_{x\in [0,1]} \left|\N(0,t+\tilde s;x) -\N(0,t; x)\right|\Big| \\
&\le \sup_{x\in [0,1]} \left|\N(0,t+s; x) -\N(0,t+\tilde s; x)\right|,
\end{align*}
and just as in the above lemma we have for $\tilde s\le  s\le 1$
\[\mP_0\left(\sup_{x\in [0,1]} \left|\N(0,t+s,x) -\N(0,t+\tilde s,x)\right|>\lambda\right) \le \exp\left(- \frac{C_2^2\lambda^2}{\sqrt{s-\tilde s}}\right). \]
Therefore a chaining argument gives us Proposition \ref{prop:n:chain}. \qed

Now we return to the case of soft obstacles. 
We will need the following
\begin{prop}\label{prop:soft}
Fix any $\eta>0$. There are constant $0<C_6(\eta)<1$ and $C_7(\eta)>0$ such that for $t\ge L$ we have 
\begin{align*} &\mP_0\bigg(\car\left(\N(0,t)\right) \le \frac{a}{8},\; \sup_{s\le C_6a^{4+\eta}} \sup_{x\in [0,1]} \left|\N(0,t+s,x)-\N(0,t,x)\right|\le \frac{a}{16},\\
&\hspace{8cm}\text{and } \sup_{s\le C_6 a^{4+\eta}}|\X_{t+s}- \X_{t}|\le \frac{a}{16}\bigg) \ge \frac{a^2}{C_7}.\end{align*}
\end{prop}
\begin{proof} We first give a lower bound on $\mP_0\left(\car(\N(0,t)) \le \frac{a}{2}\right)$. Clearly $\car(\N(0,t)) = \sup_{x,y}|\N(t;x,y)|$, where $\N(t;x,y)$ is defined in \eqref{eq:ntxy}. We have 
\[\N(t;x,y) = \N^{(1)}(t; x,y)-\N^{(2)}(t;x,y).\] 
Lemma \ref{lem:u:incr} gives
\[ \mP_0\left(\sup_{x,y\in [0,1]}|\N^{(2)}(t;x,y)| \ge \frac{a}{16}\right) \le 2 \exp\left(- \frac{e^ta^2}{256 C_1}\right).\]
 We next obtain an upper bound on the tail probabilities of $\sup_{x,y\in [0,1]}|\N^{(1)}(t;x,y)|$. Using ideas analogous to Proposition \ref{prop:n:chain}, but now we use Lemma \ref{lem:G:space} instead of the bounds on \eqref{eq:z:1} and \eqref{eq:z:2}, we obtain 
\[ \mP_0 \left(\sup_{x,y\in [0,1]}|\N^{(1)}(t;x,y)| \ge \frac{a}{16}\right)  \le \exp\left( -\frac{a^2}{256 C_4}\right)\]
for some $C_4>0$. Therefore
\begin{equation}\label{eq:r:ot}\begin{split}
\mP_0\left(\car\left(\N(0,t)\right) \le \frac{a}{2}\right) & = 1- \mP_0\left(\car\left(\N(0,t)\right) > \frac{a}{4}\right) \\
&\ge 1- \mP_0\left(\sup_{x,y\in [0,1]}|\N^{(1)}(t;x,y)| \ge \frac{a}{16}\right)  \\
&\qquad -\mP_0\left(\sup_{x,y\in [0,1]}|\N^{(2)}(t;x,y)| \ge \frac{a}{16}\right) \\
&\ge \frac{a^2}{C_5}.
\end{split}\end{equation}
for some $C_5>0$.
By Proposition \ref{prop:n:chain} and standard results on Brownian motion the quantities
\[ \mP_0\left(\sup_{s\le C_6 a^{4+\eta}} \sup_{x\in[0,1]}\left|\N(0,t+s; x)-\N(0,t; x)\right| \le \frac{a}{16} \right)\]
and 
\[ \mP_0\left(\sup_{s\le C_6 a^{4+\eta}}|\X_{t+s}- \X_{t}|\le \frac{a}{16}\right)\]
are both at least $1- \exp\left(-\frac{C}{\sqrt{C_6}a^{\frac{\eta}{2}}}\right)$, {\it uniformly} in $t$. The proposition is proved by combining the above with \eqref{eq:r:ot}.
\end{proof}

\section{The proof of upper bound in Theorems \ref{thm1} and \ref{thm2}} \label{sec:hard}

As explained earlier (in Remark \ref{rem:scaling}) due to the scaling relations, we will first obtain an upper bound for $S_T^{H,1,\nu}$.  We will consider the hard obstacle case first and then modify its proof suitably to handle the soft obstacle case.

\begin{proof}[Proof of Upper bound in Theorem \ref{thm1}] Recall from \eqref{eq:pf:hard} that \[S_T^{\HH,1,\nu} = \E_0\exp\left( -\nu \left|\mathscr{S}^1_{T}(a)\right|\right) \]
where $\mathscr{S}^1_{T}(a)$ is the sausage of radius $a$ around $\uv$, that is 
\bes \mathscr{S}^1_{T}(a) =  \mathop{\bigcup}_{\substack{0 \leq s\leq T,\\ 0 \leq y \leq 1}} \left\{ \uv(s,y) + B(\mathbf{0},a) \right\}.\ees
The upper bound on $S_T^{H,1,\nu}$ essentially involves finding a lower bound on the volume of the sausage. 

Recall from \eqref{eq:nt} that $\#(T) =\left|\left\{i\ge 1: \tau_i \le T\right\}\right|,$ counts the number of $\tau_i$'s before time $T$, and for $\Lambda >1$ from Lemma \ref{lem:nt:tail} that there are positive constants $ A_1(d, \Lambda),\, B_1(d, \Lambda)$ such for all $T>0$ 
\be \label{eq:n:t} \mP_0\left(\#(T)\le A_1 T^{\frac{d}{d+2}}\right) \le \exp\left(-B_1T^{\frac{d}{d+2}}\right).\ee
Therefore 
\begin{align*}
& \E_0\exp\left(-\nu \left|\mathscr{S}^1_{T}(a)\right|\right) \\
&\le \exp\left(-B_1 T^{\frac{d}{d+2}}\right) + \E_0\left[\exp\left(-\nu  \left|\mathscr{S}^1_{T}(a)\right|\right)\cdot \mathbf{1} \left\{\#(T) > A_1 T^{\frac{d}{d+2}}\right\}\right].
\end{align*}
Now let 
\[ \#_1(T) :=\left|\left\{i\le A_1 T^{\frac{d}{d+2}}:\, \tau_{i+1}-\tau_i\ge \Lambda^2\right\}\right|.\]
Clearly $\tau_{i+1}-\tau_i$ is more than the time it takes for the Brownian motion $\X_t$ to leave a ball of radius $4\Lambda$ centered at $\X_{\tau_i}$. Therefore the sequence $\tau_{i+1}-\tau_i$ stochastically dominates an i.i.d. sequence $\mathcal T_i$, where $\mathcal T_i$ is distributed as the time it takes for a Brownian motion starting at $\mathbf 0$ to leave a ball of radius $4\Lambda$. Moreover 
\[ \mP\left( \mathcal T_i \ge \Lambda^2\right) =p>0,\]
where $p$ is independent of any of the parameters. Therefore by standard large deviation theory, there are positive constants $A_2(p,d, \Lambda),
B_2(p,d,\Lambda)$ such that 
\be \label{eq:n1:t} \mP_0\left(\#_1(T) < A_2 T^{\frac{d}{d+2}}\right) \le \exp\left(-B_2 T^{\frac{d}{d+2}}\right).\ee
Consider the event
\[ \mathscr A_1:= \left\{\#(T)> A_1T^{\frac{d}{d+2}},\quad \#_1(T)> A_2 T^{\frac{d}{d+2}}\right\}. \]
Equations \eqref{eq:n:t} and \eqref{eq:n1:t} imply that there is a positive $B_3(d, \Lambda, p)$ such that 
\be \label{eq:a1}
\mP_0\left(\mathscr{A}_1^c\right) \le \exp\left(-B_3 T^{\frac{d}{d+2}}\right).
\ee
To get a good lower bound on the volume of the sausage, we need a good control on the number of $T_i$'s up to time $T$. So let 
\[ \#_2(T):= \left|\left\{T_i: T_i \le T\right\}\right|.\]
Our next task is to show that on the event $\mathscr A_1$ we must have that $\#_2(T)$ is sufficiently large. 

On the event 
\[ \mathscr{A}_2= \left\{\sum_{i=1}^{A_4 \frac{T^{\frac{d}{d+2}}}{L}} (S_i-T_{i-1}) \le  \widetilde{C} A_4T^{\frac{d}{d+2}}\right\}, \]
the number of $\tau_i$ with $\tau_{i+1}-\tau_i \ge \Lambda^2$ in the union of all the intervals $[T_{i-1}, S_i],\, i=1,2\cdots, A_4\frac{T^{\frac{d}{d+2}}}{L}$ is less than $\frac{\widetilde{C}A_4 T^{\frac{d}{d+2}}}{\Lambda^2}$. On the event $\mathscr A_1$ we have $\#_1(T) > A_2T^{\frac{d}{d+2}}$. Therefore with the choice of $A_4$ such that 
\be \label{eq:a4} \frac{\widetilde{C}A_4}{\Lambda^2} = \frac{A_2}{4}, \ee
the intersection of $\mathscr A_1$ and $\mathscr{A}_2$ contains at least $\frac{ A_2}{2}T^{\frac{d}{d+2}}$ many $\tau_i$'s before time $T$ outside of the union of intervals $[T_{i-1}, S_i],\, i=1,2\cdots, A_4\frac{T^{\frac{d}{d+2}}}{L}$. Note carefully that any $T_j$ is the smallest $\tau_i$ immediately following $S_j$, and therefore this guarantees that there are at least $A_4\frac{T^{\frac{d}{d+2}}}{L}$ many $T_j$'s up to time $T$. Then using Lemma \ref{lem:eAla} and \eqref{eq:a1} we have that 
\be\label{eq:n2:t}
\mP_0\left(\#_2(T) <A_4\frac{T^{\frac{d}{d+2}}}{L} \right) \le \mP_0(\mathscr{A}_1^c)  +\mP_0(\mathscr{A}_2^c)  \le \exp\left(-B_4 T^{\frac{d}{d+2}}\right),\ee
for some $B_4(d,\Lambda, p)>0$.

Therefore on the event $\mathscr A_3:=\mathscr A_1 \cap \mathscr A_2$ we have that 
$\#_2(T) \ge A_4\frac{T^{\frac{d}{d+2}}}{L}.$
We now count the $T_i$'s before time $T$ with large sausage volumes at $T_i's$, as in  Lemma \ref{lem:ti}. Namely,
\bes 
\#_3(T) := \left|\left\{ i \le A_4\frac{T^{\frac{d}{d+2}}}{L}:\,\car\left(\N(T_{i-1}, T_i)\right) \le \Lambda,\; \left |\cs\Big(\frac{a}{2};\,\N(T_{i-1}, T_i))\Big) \right| \ge {\mathcal C}_\gamma a^{d-2+\gamma} \right\}\right|.
\ees

It now follows from \eqref{eq:ti}, for $\Lambda >1$ as in Lemma \ref{lem:r:unif}, that for any $A_5>0$
\begin{align*}
\mP_0\left(\#_3(T)\le  A_5\frac{T^{\frac{d}{d+2}}}{L}\right) & = \mP_0 \left(\exp\left(-\#_3(T)\right) \ge \exp\left(-  A_5\frac{T^{\frac{d}{d+2}}}{L}\right)\right) \\
&\le \left(\frac12+e^{-1}\right)^{A_4\frac{T^{\frac{d}{d+2}}}{L}}\exp\left(A_5\frac{T^{\frac{d}{d+2}}}{L} \right)\\
&\le \exp\left(A_4\frac{T^{\frac{d}{d+2}}}{L}\log \left(\frac12+e^{-1}\right) +A_5 \frac{T^{\frac{d}{d+2}}}{L}\right).
\end{align*}
We now choose
\[ A_5= -\frac{A_4}{2}\log\left(\frac12+e^{-1}\right),\]
to obtain 
\be \label{eq:n3:t} \mP_0\left(\#_3(T)\le  A_5\frac{T^{\frac{d}{d+2}}}{L}\right)  \le \exp\left(- B_5\frac{ T^{\frac{d}{d+2}}}{L} \right),\ee
for a positive constant $B_5(d, \Lambda, p)$.

 Let $n_1<n_2<\cdots $ be the indices $j\le A_4\frac{T^{\frac{d}{d+2}}}{L}$ such that $T_j\le T$ and
\begin{equation} \label{eq:volbd}
\,\car\left(\N(T_{i-1}, T_i)\right) \le \Lambda,\; \left |\cs\Big(\frac{a}{2};\,\N(T_{i-1}, T_i))\Big) \right| \ge {\mathcal C}_\gamma a^{d-2+\gamma}.
\end{equation}
Thanks to \eqref{eq:n:t}, \eqref{eq:n1:t}, \eqref{eq:n2:t} and \eqref{eq:n3:t}, outside of a set of probability $\exp\left(-B_6\frac{T^{\frac{d}{d+2}}}{L}\right)$, where $B_6$ is a positive constant depending only on $d, \Lambda, p$, there are at least $A_5\frac{T^{\frac{d}{d+2}}}{L}$ many such $n_i$'s.  Further, from \eqref{eq:r:u-n} we have  that the {\it fixed-time} sausages $\mathscr{S}^1(a; T_{n_j})$ are disjoint.  Therefore, using \eqref{eq:vol:u-n} we have
\begin{align*}
S_T^{\HH,1,\nu} & = \E_0\exp\left( -\nu \left|\mathscr{S}^1_{T}(a)\right|\right)\\
&\le \exp\left(-B_6\frac{T^{\frac{d}{d+2}}}{L}\right) + \E_0\left[\exp\left(-\nu\left|{\bigcup}_{j=0}^{A_5\frac{T^{\frac{d}{d+2}}}{L}}\mathscr{S}^1(a; T_{n_j})\right|\right)\right] \\
&\le \exp\left(-B_6\frac{T^{\frac{d}{d+2}}}{L}\right) + \E_0\left[\exp\left(-\nu\left|{\bigcup}_{j=0}^{A_5\frac{T^{\frac{d}{d+2}}}{L}}\mathscr{S}(\frac{a}{2}; {\mathbf N}(T_{n_{j-1}},T_{n_j})\right|\right)\right] 
\end{align*}
Now applying \eqref{eq:volbd},
\begin{align*}
S_T^{\HH,1,\nu} &\le \exp\left(-B_6\frac{T^{\frac{d}{d+2}}}{L}\right) + \exp\left(-\nu A_5 C_{\gamma}a^{d-2+\gamma}\frac{T^{\frac{d}{d+2}}}{L} \right). 
\end{align*}

Finally we apply the scaling \eqref{eq:scaling} to get an upper bound for $S_T^{H, J, \nu}$. We thus obtain 
\begin{align*}
S_T^{\HH, J, \nu} &\le \exp\left(-\frac{B_6(T/J^2)^{\frac{d}{d+2}}}{E+3|\log(a/J^{\frac12})|}\right) + \exp\left(-\nu A_5 C_{\gamma}a^{d-2+\gamma}\frac{(T/J^2)^{\frac{d}{d+2}}J^{1-\frac{\gamma}{2}}}{E+3|\log(a/J^{\frac12})|} \right).
\end{align*}
It is clear that the first term is the leading term. 
\end{proof}

We now turn to the case of soft obstacles. As before we first find an upper bound for $S_T^{H,1,\nu}$. We will explain how the argument differs from the case of hard obstacles.

\begin{proof}[The proof of upper bound in Theorem \ref{thm2} ]
 We will now look for a subsequence $T_{n_j}$ of the $T_i$'s such that the entire string stays in a small ball of radius $\frac{3a}{8}$ during the time interval $[T_{n_j}, T_{n_j}+a^{4+\epsilon}]$, and there exists a Poisson point within distance $\frac{a}{8}$ of the center of mass (See Definition \ref{eq:n3:t:soft}). Assumption \ref{ass} then guarantees there is a contribution of at least $\mathscr {C} a^{4+\epsilon}$ to the integral in $S_T$ during this time period. 

We then follow the argument of Proof of upper bound in Theorem \ref{thm1}. The only difference now is in the definition of $\#_3(T)$, which in the present case becomes
\be\label{eq:n3:t:soft}\begin{split}
&\#_3(T)=\bigg\vert \bigg\lbrace i \le A_4 \frac{T^{\frac{d}{d+2}}}{L} :\car(\N(T_{j-1}, T_j)) \le \frac{a}{8},\\
 &\qquad \qquad \sup_{s\le C_6a^{4+\eta}} \sup_{x\in [0,1]} \left|\N(T_{j-1},T_{j}+s,x)-\N(T_{j-1}, T_j,x)\right|\le \frac{a}{16}, \\
&  \qquad \qquad\sup_{s\le C_6a^{4+\eta}}|\X_{T_j+s}- \X_{T_j}|\le \frac{a}{16},\; \\
&\qquad \qquad\text{and there is a Poisson point within distance $\frac{a}{8}$ of $\X_{T_j}$ }  \bigg\rbrace\bigg\vert
\end{split}
\ee
Therefore $\#_3(T)$ is a sum of Bernoulli random variables with probability of success $\tilde p$ satisfying 
\[ \tilde p\ge \frac{a^2}{C_7}  \left(1- \exp\left(-\nu \frac{a^d}{8^d}\right)\right) \ge C_8\nu a^{d+2}\]
for some constant $0<C_8<1$.

 Let $Z_1, Z_2,\cdots $ be i.i.d. Bernoulli random variables with success probability $p_*= C_8\nu a^{d+2}$. By standard large deviation theory
\be \label{eq:n3:tail:soft} \begin{split}
\mP_0\left(\#_3(T) \le A_4 \frac{T^{\frac{d}{d+2}}}{L} \cdot \frac{p_*}{2} \right) & \le \mP_0\left(\sum_{i=1}^{\frac{A_4 T^{\frac{d}{d+2}}}{L}} Z_i  \le A_4 \frac{T^{\frac{d}{d+2}}}{L} \cdot \frac{p_*}{2} \right) \\
&\le \exp\left(-B_7 \frac{T^{\frac{d}{d+2}}}{L}\cdot \nu a^{d+2}\right),
\end{split}\ee
for some constant $B_7(d,\Lambda,p)>0$.

Let $n_1<n_2<\cdots $ be the indices $j$ such that the event in the right hand side of \eqref{eq:n3:t:soft} occurs. 
Note that at these times $T_{n_k}$ it follows from the proof of \eqref{eq:r:u-n}
\begin{align*}
 \sup_{s\le C_6a^{4+\eta}}  \car(\uv(T_{n_k}+s)) &\le \sup_{s\le C_6a^{4+\eta}} \car(\N(T_{n_k-1}, T_{n_k}+s))+2\delta \\
 & \le \car(\N(T_{n_k-1}, T_{n_k})) +\frac{a}{8} +2\delta \\
 & \le  \frac{a}{4} +2\delta 
\end{align*}
Moreover we have that the center of mass satisfies $\sup_{s\le C_6 a^{4+\eta}}\left|\X_{T_{n_k}+s} - \X_{T_{n_k}}\right|\le \frac{a}{16}$. Therefore 
\begin{align*} \sup_{s\le C_6a^{4+\eta}} &\left| \uv(T_{n_k}+s, x) - \X_{T_{n_k}} \right|   \\
&\le \sup_{s\le C_6 a^{4+\eta}} \left|\uv(T_{n_k}+s, x) - \X_{T_{n_k}+s}\right| +\sup_{s\le C_6 a^{4+\eta}} \left|\X_{T_{n_k}+s}-\X_{T_{n_k}}\right|   \\
&\le \sup_{s\le C_6 a^{4+\eta}}  \car(\uv(T_{n_k}+s)) +\sup_{s\le C_6 a^{4+\eta}}\left|\X_{T_{n_k}+s}-\X_{T_{n_k}}\right|  \\
&\le \frac{5a}{16} +2\delta \\
&\le \frac{3a}{8}.
\end{align*}
Thus the entire string lies within a ball of radius $\frac{3a}{8}$ centered at $\X_{T_{n_k}}$ for the duration $[T_{n_k}, T_{n_k}+C_6 a^{4+\eta}]$. Since there is a Poisson point within distance $\frac{a}{8}$ of $\X_{T_{n_k}}$  the string will be entirely contained within distance $\frac{a}{2}$ of the Poisson point during the time interval $[T_{n_k}, T_{n_k}+C_6a^{4+\eta}]$. Therefore using the bounds in the proof of the upper bound in Theorem \ref{thm1} along with \eqref{eq:n3:tail:soft} we obtain for $T>0$
\begin{align*}
S_T^{\HH, 1,\nu}&=\E\left[\exp\left(- \int_0^T\int_{0}^1 \V(\uv(s,x),\text{\Large$\eta$}) ds dx\right)\right] \\
&\le \exp\left(-B_7 \frac{T^{\frac{d}{d+2}}}{L}\cdot \nu a^{d+2}\right) +\exp\left(-B_8 \frac{T^{\frac{d}{d+2}}}{L} \right)\\
&\quad  + \E\left[\exp\left(- \sum_{j=1}^{\frac{A_5}{L} T^{\frac{d}{d+2}}\nu a^{d+2}}\int_{T_{n_j}}^{T_{n_j}+C_6 a^{4+\eta}}\int_{0}^1 \V(\uv(s,x),\text{\Large$\eta$}) ds dx\right)\cdot \right],
\end{align*}
where $A_5= \frac{A_4 C_8}{2}$. Now we use Assumption \ref{ass} to obtain an upper bound
\begin{align*}
S_T^{\HH, 1,\nu}&\le  \exp\left(-B_7 \frac{T^{\frac{d}{d+2}}}{L}\cdot \nu a^{d+2}\right) +\exp\left(-B_8 \frac{T^{\frac{d}{d+2}}}{L} \right) \\
&\qquad + \exp\left(- \mathscr CA_5 C_6\nu a^{d+6+\eta} \frac{T^{\frac{d}{d+2}}}{L}\right) 
\end{align*}

Finally we use \eqref{eq:scaling} to get an upper bound for $S_T^{\HH, J,\nu}$. We obtain for $T>0$
\begin{align*}
S_T^{\HH, J,\nu} & \le \exp\left(- \frac{B_7 \nu a^{d+2}(T/J^2)^{\frac{d}{d+2}}}{J\left(E+ 3 |\log (a/J^{\frac12})|\right)}\right) 
+\exp\left(- \frac{\mathscr CA_5 C_6\nu a^{d+6+\eta} (T/J^2)^{\frac{d}{d+2}}}
{J^{3+\frac{\eta}{2}}\left(E+ 3 |\log (a/J^{\frac12})|\right)}\right). 
\end{align*}
It is clear that the second term dominates the first term for large $J$. 

\end{proof}

\bibliography{polymer}
\bibliographystyle{amsalpha}

\end{document}